\documentclass[12pt, twoside]{article}

\usepackage{amsmath}
\usepackage{amsthm,amssymb}
\usepackage{times}
\usepackage{enumerate}

\theoremstyle{definition}
\newtheorem{thm}{Theorem}[section]
\newtheorem{algo}[thm]{Algorithm}
\newtheorem{prop}[thm]{Proposition}
\newtheorem{cor}[thm]{Corollary}
\newtheorem{lem}[thm]{Lemma}
\newtheorem{notn}[thm]{Notation}
\newtheorem{defn}[thm]{Definition}
\newtheorem{rem}[thm]{Remark}

\numberwithin{equation}{section}

\newcommand{\subjclass}[1]{\bigskip\noindent\emph{2010 Mathematics Subject Classification:}\enspace#1}
\newcommand{\keywords}[1]{\noindent\emph{Keywords:}\enspace#1}

\usepackage[backend=bibtex,maxnames=4]{biblatex}

\bibliography{lit.bib}
\DefineBibliographyStrings{english}{%
  references = {Bibliography},
}

\usepackage[colorlinks=true,
linkcolor=black,citecolor=black,filecolor=black,urlcolor=black,menucolor=black]{hyperref}

\begin{document}


\baselineskip=17pt


\def \h{\sqrt{h}}

\def \N{\mathbb{N}}
\def \R{\mathbb{R}}
\newcommand {\sign} {\mathop \mathrm{sign}} 
\newcommand {\trace} {\mathop \textup{tr}}
\def \ae{\text{a.\ e.}}
\def \diver{\nabla_\textup{tan}\cdot}
\newcommand {\lip} {\mathop \mathrm{Lip}}
\newcommand {\dist} {\mathop \mathrm{dist}}
\def \sphere{{{S}^{d-1}}}
\def \numphases{P}
\newcommand {\supp} {\mathop \textrm{supp}}
\def \chara{\mathbf{1}}
\def \gap{\varepsilon^2} 
\def \diss{\textup{diss}}
\def \process{\mathscr{P}}
\def \quotient{Q}
\def \kmax{K}
\def \lmax{L}
\def \nmax{N}
\def \kmean{\frac1\kmax \sum_{k=1}^\kmax}
\def \lmean{\frac1\lmax \sum_{l=1}^{\lmax}}
\def \nmean{\frac1\nmax \sum_{n=1}^\nmax}
\def \tmean{\frac1T \int_0^T}
\newcommand{\textoverline}[1]{$\overline{\mbox{#1}}$}
\def \gap{\varepsilon^2}

\newcommand{\step}[1]{\par\medskip\par\noindent\textit{#1}}

\title{Convergence of thresholding schemes\\
incorporating bulk effects}

\author{Tim Laux\\
Max Planck Institute \\for Mathematics in the Sciences\\
Inselstra\ss e 22\\
04103 Leipzig\\
tim.laux@mis.mpg.de\\\\
Drew Swartz\\
Booz Allen Hamilton\\
drew.e.swartz@gmail.com}

\date{\today}

\maketitle


\begin{abstract}
In this paper we establish the convergence of three computational algorithms for interface
motion in a multi-phase system, which incorporate bulk effects.  
The algorithms considered fall under the classification 
of thresholding schemes, in the spirit of the celebrated Merriman-Bence-Osher
algorithm for producing an interface moving by mean curvature.
The schemes considered here all incorporate either a local force coming from an energy in the bulk,
or a non-local force coming from a volume constraint.
We first establish the convergence of a scheme proposed by Ruuth-Wetton for approximating
volume-preserving mean-curvature flow.  Next we study a scheme for the geometric flow
generated by surface tension plus bulk energy.  Here the limit is motion by mean curvature (MMC) plus forcing term.
Last we consider a thresholding scheme for simulating grain growth in a polycrystal surrounded by air, which 
incorporates boundary effects on the solid-vapor interface. 
The limiting flow is MMC on the inner grain boundaries, 
and volume-preserving MMC on the solid-vapor interface.

\subjclass{Primary 65M12; Secondary 35A15.}

\keywords{Mean curvature flow; Thresholding; MBO scheme; Minimizing movements; Volume preserving}
\end{abstract}

\section*{Introduction}
Surfaces moving with curvature dependent velocities is a phenomenon of interest
in many physical models.
A standard model of this type of interfacial motion is mean-curvature flow, which appears as the effective evolution equation
of grain boundaries in Mullins' model \cite{mullins1956two} or as the singular limit of the Allen-Cahn equation describing the evolution of antiphase boundaries \cite{allen1979microscopic}.
The motion law then is $V = H$, where $V$ denotes the velocity in the normal direction, 
and $H$ is the scalar mean curvature of the interface. It is a system of degenerate parabolic equations and can be regarded
as the gradient flow of the interfacial energy
w.\ r.\ t.\ the $L^2$-metric on the space of normal velocities.
A similar motion is observed in multi-phase systems where the energy depends on a possibly weighted
sum of the interfacial energies between the phases.
This is a prominent model for grain growth in polycrystals, where each phase
represents a grain, i.\ e.\ a part of the volume with homogeneous crystal structure \cite{mullins1956two}.

If one also considers bulk energies in the model, a forcing term arises in the velocity, 
leading to the equation $V = H + f$, where the force $f$ might in general be non-local.
A particular example of a non-local forcing arises when
the volume of the bulk is constrained to stay constant. 
This leads to
\emph{volume-preserving mean-curvature flow}.
Here the motion law is $V =  H - \langle H \rangle$,
where $\langle \cdot \rangle$ denotes the average over the interface.
This evolution arises for example in the modeling of metallic alloys
or Ostwald ripening, a process describing
the change in inhomogeneous structure of a dispersion.
Coarsening is observed in these processes, and the coarsening rates can be measured by collecting statistical
data from a series of experiments or numerical simulations.

For this and other purposes it is desirable to have efficient computational schemes for
producing various types of curvature driven flows.
In this paper we will examine computational models 
for the examples described above, namely,
volume-preserving motion by mean curvature, 
motion by mean curvature with a local forcing term
and a model for grain growth in polycrystals incorporating boundary effects.
The main results of this paper are the rigorous
convergence results for these algorithms, Theorems \ref{thm1}, \ref{thm2} and \ref{thm3}.

\medskip

The class of algorithms we consider are so-called \emph{thresholding algorithms}.
The idea goes back to Merriman, Bence and Osher, who introduced a nowadays highly
appreciated time discretization to generate motion by mean curvature in \cite{MBO92}.
This algorithm has colloquially become known as
the \emph{MBO scheme}. 
It is based on a time splitting for a slow-reaction fast-diffusion
process in order to bypass the numerical difficulty of multiple scales.
Starting from the phase $\Omega^0$, i.\ e.\ an open, bounded set in $\mathbb{R}^d$, 
 with characteristic function $\chara_{\Omega^0}$, one solves 
the heat equation with initial data $\chara_{\Omega^0}$ for a short time $h>0$,
i.\ e. one defines the function $\phi := G_h \ast \chara_{\Omega^0}$,
where $G_h$ denotes the heat kernel at time $h$. 
One then updates to the evolved phase $\Omega^1$ by thresholding $\phi$ at the 
value $\tfrac12$, i.\ e. taking $\Omega^1$ to be the super level set $\{\phi>\tfrac12\}$.
The procedure is then repeated with the updated set.
This scheme produces a discrete sequence of interfaces $ \Sigma^h(nh) \equiv \Sigma^n= \partial \Omega^n $.

It has been shown that MBO dynamics converge to motion by mean curvature
as $h \rightarrow 0^+$.
Rigorous convergence proofs have been established independently by Evans \cite{evans1993convergence} and
Barles and Georgelin \cite{barles1995simple}.
Their proofs rely on the fact that the scheme preserves a structural feature of mean-curvature flow,
a geometric comparison principle. This allows the authors to use the level set formulation of mean-curvature flow
which can be treated using the theory of viscosity solutions for second-order parabolic PDE.
However, a number of extensions of the MBO scheme for different curvature driven motions 
have been developed that do not satisfy a comparison principle; see e.\ g.\  
\cite{bonnetier2012consistency,EseOtt14,esedoglu2008threshold,  ishii1999threshold, RutWet03}.
This is not a weakness of these algorithms but inherent in the equations.
The convergence proofs in \cite{barles1995simple,evans1993convergence} do not apply in these cases.

Two more recent proofs have established the convergence of MBO, but do not rely on a comparison
principle.
Using asymptotic techniques, Yip and the second author \cite{SwaYip15} established 
a short-time convergence result along with quantitative properties 
such as convergence rate and bounds on curvature growth.  
Otto and the first author \cite{LauOtt15} established a conditional long-time convergence result
also for the case of multiple phases by exploiting the gradient flow structure.
In this paper, we show how to adapt the proof of the second approach \cite{LauOtt15} to the situations mentioned above.

\medskip

Ruuth and Wetton \cite{RutWet03} extended the thresholding scheme to produce an interface
moving by volume-preserving mean-curvature flow. Here one simply changes
the threshold parameter from $\tfrac12$ to the value $\lambda \in (0,1)$
such that the volume is preserved, i.\ e.\ $\left | \{ \phi > \lambda \} \right| = | \Omega^0 |$.  
In Section \ref{sec:vol} we will provide a convergence proof for this scheme, cf. Theorem \ref{thm1}.

The inspiration for changing the threshold value comes from
Mascarenhas in \cite{mascarenhas1992diffusion} who simulates an affine forcing
term. He observes that changing the threshold value from $\tfrac12$
to $\tfrac12 - \tfrac{f}{2\sqrt{ \pi}} \sqrt{h}$  seems to produce approximate solutions to
$V=H+f$ for a constant force $f$.  In Section \ref{sec:force} we adapt this idea to produce a 
thresholding scheme for interfaces moving by mean curvature plus a local forcing term, i.\ e.\
$V = H + f$ with a space-time dependent force $f=f(x,t)$.  
In addition we give a convergence proof of this scheme in Theorem 2.

The above mentioned schemes extend naturally to multi-phase motions if one assumes equal surface
tensions between the phases, cf. \cite{MBO94}.
The extension to arbitrary surface tensions by
Esedo\u{g}lu and Otto in \cite{EseOtt14} is less obvious and comes from an energetic view-point on which we will comment in the next paragraph.
In \cite{elsey2011large}, Elsey, Esedo\u{g}lu and Smereka use
the multi-phase schemes to perform large-scale computational
simulations for grain growth in polycrystals.
Convergence of the algorithm in \cite{EseOtt14} was recently established
in \cite{LauOtt15}. 
These simulations assume periodic boundary conditions and are therefore restricted to the interior behavior in a polycrystal.
Taking into account boundary effects on the solid-vapor interface is more difficult.
It is known that the outer boundary of a polycrystal moves by \emph{surface diffusion}, which is a fourth order flow.
However, computational simulations involving fourth order flows present various challenges.
In Section \ref{sec:crystal} we discuss a simpler algorithm proposed by Esedo\u{g}lu and Jin in \cite{An15} for approximating
these effects.  
They consider an algorithm which replaces surface diffusion, the fourth order local motion law on the outer boundary of the polycrystal, by volume-preserving
mean-curvature flow,
a second order but non-local equation.
This is plausible because both motions are
volume preserving and (due to the gradient flow structure) energy dissipative flows for the area functional. Simulations for this model have been performed in \cite{An15}, 
demonstrating that the model is reasonable and captures the typical effect of surface grooving.  However it is admittedly not perfect, as it is also shown that for large
numbers of grains ($\sim 10^3$), non-physical phenomenon are observed in the simulations.
In Theorem \ref{thm3} we show that the proof in \cite{LauOtt15} can also be applied in this situation under some moderate modeling assumptions.  
The limiting motion is shown to be mean-curvature flow on the inner
grain boundaries, and volume-preserving mean-curvature flow on the outer boundary of the whole polycrystal.

\medskip

The basis of our proofs is the interpretation of the MBO scheme as a \emph{minimizing movements} scheme by Esedo\u{g}lu and Otto in \cite{EseOtt14}.
Minimizing movements is a natural time-discretization of a gradient flow which can be seen as a generalization of the implicit Euler scheme. 
It was introduced by De Giorgi in the general framework \cite{de1993new} and for mean-curvature flow by Almgren, Taylor and Wang in
\cite{ATW93} and Luckhaus and Sturzenhecker in \cite{LucStu95}. 
Let us elaborate more on the connection between thresholding schemes and minimizing movements drawn in \cite{EseOtt14} in the 
case of two phases. The functional $E_h(\chi)=\frac1\h \int \left(1-\chi \right) G_h\ast \chi \,dx$
is an approximation of the perimeter of the set $\{\chi=1\}$. 
Indeed, it was shown in \cite{miranda2007short} and later on with different techniques in \cite{EseOtt14} that these functionals $\Gamma$-converge to 
$E(\chi) = \frac1{\sqrt{\pi}} \int \left| \nabla \chi\right|$ as $h\to0$. 
It is the case that MBO is equivalent to running minimizing movements for dissipating $E_h$, 
where $D_h(\omega) = \frac1\h \int \omega \,G_h\ast \omega \, dx$ is the metric term penalizing distances between two sets.
More specifically, starting with an initial set $\Omega^0\subset \R^d$, 
setting $\chi^0:= \chara_{\Omega^0}$ to be the characteristic function of this set, it turns out that the sets $\Omega^n = \{\chi^n =1 \}$ 
generated by the MBO scheme can be characterized by
\begin{align*}
\chi^n = \arg\min \left\{ E_h(\chi^n) + D_h(\chi^n-\chi^{n-1})\right\}.
\end{align*}
This allows for energetic techniques used in the study of gradient flows. 
We show in Lemma \ref{la_MM} that this structural property is conserved in the case of the scheme for volume-preserving mean-curvature flow
in \cite{RutWet03}. In particular, we have the important a priori estimate (\ref{ED-estimate}).
Most recently Mugnai, Seis and Spadaro \cite{mugnai2015global} 
studied a volume-preserving variant of the above mentioned minimizing movements scheme \cite{ATW93, LucStu95} and proved a conditional
convergence result in the same way as Luckhaus and Sturzenhecker. In the proof of Theorem \ref{thm1} we face similar issues as the ones in that work. 
Bellettini, Caselles, Chambolle and Novaga \cite{bellettini2009volume} studied anisotropic versions of mean-curvature flow starting from \emph{convex} sets.
In particular they proved convergence of the thresholding scheme with uniformly bounded forcing terms. 
Furthermore, they considered a variant of the volume-preserving scheme \cite{RutWet03}
where the volume is not precisely preserved in the approximation but still in the limit when the time-step size goes to zero.
They are able prove uniform bounds on the resulting forcing term. In contrast, we work with the exact constraint on the volume
and only work with an $L^2$-bound on the forcing term coming from the Lagrange multipliers associated to the volume constraint.
We establish this bound in Proposition \ref{L2 bounds lambda}.
In Lemma \ref{1d lemma} we generalize the one-dimensional estimate Lemma 4.2 and Corollary 4.3 in \cite{LauOtt15} to our situation where the threshold 
value may differ from $\frac12$.
A common thread in the above mentioned works \cite{LauOtt15,LucStu95, mugnai2015global}, and in ours as well, is an area-convergence assumption, here (\ref{conv_ass}).
This assumption prevents a sudden loss of interfacial area as the time step tends to zero which is not guaranteed by the a priori estimate (\ref{ED-estimate}).
It is an interesting task to validate this assumption, even for the classical MBO scheme, under convexity assumptions on the initial phase.

\section{Volume-preserving mean-curvature flow}\label{sec:vol}
In this section, we discuss a scheme for volume-preserving motion by mean curvature, here Algorithm \ref{MBO_volume}, 
which was introduced by Ruuth and Wetton in \cite{RutWet03}.
We first state the algorithm and fix the notation, and present the main result of this section in Theorem \ref{thm1}. 
Following this we give the details of the proof of the theorem.

\subsection{Algorithm and notation}
\begin{algo}\label{MBO_volume}
Given the phase $\Omega$, i.\ e.\ an open, bounded set in $\R^d$, with $|\Omega|=1$ at time $t=(n-1)h$, 
obtain the evolved phase $\Omega'$ at time $t=nh$ by:
\begin{enumerate}
\vspace{-3pt}
 \item Convolution step:
$
 \phi := G_h\ast \chara_{\Omega}.
$
  \item Defining threshold value: Pick $\lambda$ such that 
  $
  \left| \left\{ \phi >  \lambda \right\} \right| = 1.
  $
  \item Thresholding step:
  $
  \Omega' := \left\{ \phi >  \lambda \right\}.
  $
\end{enumerate}
\end{algo}
Here and throughout the paper
\begin{align*}
	G_h(z) := \frac{1}{(4\pi h)^{d/2}} \exp\left(-\frac{|z|^2}{4h}\right)
\end{align*}
denotes the heat kernel at time $h$.

\begin{rem}
In general, the threshold value $\lambda$ is not necessarily a regular value of $\phi$,
so that a priori we cannot say that the function  $s\mapsto \left| \left\{ \phi >  s \right\} \right|$
will attain the value $1$ for any $s\in [0,1]$. 
Since by Sard's Lemma a.\ e.\ value of $\phi$ is a regular value, this practically does not happen in simulations.
Therefore, as in \cite{RutWet03}, we ignore this fact in stating the algorithm.
Our analysis also works if one replaces the second step of the scheme by defining $\lambda$ via 
$$
\lambda:= \inf \{s>0 \colon \left| \left\{ \phi >  s \right\} \right|<1\}
$$
and then chooses the updated set in the following way:
$$
\left\{ \phi >  \lambda \right\}\subset \Omega'\subset \left\{ \phi \geq  \lambda \right\}
\quad
\text{such that }
\quad
 \left| \Omega' \right| = 1.
$$
\end{rem}
\begin{notn}
We denote the characteristic function of  $\Omega^n$ at the $n$-th time step
by $\chi^n$, i.\ e.
\begin{align*}
 \chi^n := \chara_{\Omega^h}\big|_{t=nh} \equiv  \chara_{\Omega^n}
\end{align*}
and interpolate these functions piecewise constantly in time, i.\ e.
\begin{align*}
 \chi^h(t) := \chi^n\quad \text{for } t\in [nh,(n+1)h).
\end{align*}
As in \cite{EseOtt14}, here for the two-phase case, we define the following approximate energies
\begin{align}\label{def_E}
  E_h(\chi) := \frac1{\sqrt{h}}  \int \left(1-\chi\right) G_h\ast  \chi \,dx,
\end{align}
for $\chi\colon \R^d\to \{0,1\}$
and the approximate dissipation functionals as
\begin{align}\label{def_D}
 D_h(\omega) := \frac1{\sqrt{h}}  \int \omega \,G_h\ast \omega \,dx
\end{align}
for any $\omega\colon \R^d \to \{-1,0,1\}$.
\end{notn}
\begin{rem}
As $h\to0$, the approximate energies $E_h$ $\Gamma$-converge to the perimeter functional
\begin{align*}
 E(\chi) := \frac1{\sqrt{\pi}} \int \left| \nabla\chi\right|
\end{align*}
 w.\ r.\ t. the $L^1$-topology.
Esedo\u{g}lu and Otto proved in \cite{EseOtt14} that this $\Gamma$-convergence which has already been established by Miranda et.\ al. in 
\cite{miranda2007short}
is a consequence of pointwise convergence of the functionals, namely
\begin{align}\label{EhtoE}
 E_h(\chi)\to E(\chi) \quad \text{for any } \chi\in\{0,1\},
\end{align}
and the following approximate monotonicity: For any $0<h\leq h_0$ and any $\chi\in \{0,1\}$,
\begin{align}\label{appr_mon}
 E_h(\chi) \geq \left( \frac{\sqrt{h_0}}{\h + \sqrt{h_0}} \right)^{d+1}E_{h_0}(\chi).
\end{align}
\end{rem}

Our main result of this section, Theorem \ref{thm1}, establishes the convergence of the scheme towards the following weak formulation
of volume-preserving mean-curvature flow which was also used by Mugnai, Seis and Spadaro \cite{mugnai2015global} and is the analogue of 
the formulation used by Luckhaus and Sturzenhecker without the volume constraint \cite{LucStu95}.

\begin{defn}[Volume-preserving motion by mean curvature]\label{def_motion_by_mean_curvature}
 We say that
$
\chi: (0,T)\times \R^d
\to \{0,1\} 
$
is a \emph{solution to the volume-preserving mean-curvature flow equation with initial data $\chi^0$} 
if there exists a function
$V\colon (0,T)\times\R^d\to \R$ with $V\in L^2(\left| \nabla \chi\right|dt)$ such that
\begin{align}\label{H=v}
\int_0^T \int \left( \nabla\cdot \xi - \nu\cdot \nabla \xi \, \nu\right) \left|\nabla\chi\right|  dt
 = \int_0^T \int \left(V + \Lambda\right) \, \xi \cdot \nu \left| \nabla \chi \right| dt
\end{align}
for any $\xi \in C_0^\infty((0,T)\times \R^d)$ and
\begin{align}\label{v=dtX}
\int_0^T \int \partial_t \zeta \, \chi\, dx\,dt + \int \zeta(0)\, \chi^0 \,dx
= - \int_0^T \int \zeta \,V\left| \nabla \chi\right|dt
\end{align}
for all $\zeta\in C_0^\infty([0,T)\times \R^d)$,
where $\Lambda \in L^2 (0,T) $ is the average of the generalized mean curvature $H\in L^2(\left| \nabla \chi\right|dt)$ of $\chi$:
\begin{align}\label{average mean curvature}
 \Lambda := \langle H \rangle = \frac{\int H \left| \nabla \chi \right|}{\int \left| \nabla \chi \right|} .
\end{align}
\end{defn}
\begin{rem}
For our convergence proof we assume the following convergence of the energies which is not guaranteed by the a priori estimates we have at hand:

\begin{align}\label{conv_ass}
 \int_0^T E_h(\chi^h)\,dt \to \int_0^T E(\chi)\,dt.
\end{align}
\end{rem}
In the following we prove Theorem \ref{thm1} using the techniques from \cite{LauOtt15}. 
Throughout this section, we write $A\lesssim B$ if there exists a constant $C=C(d)<\infty$ such that $A\leq C B$.
Combining (\ref{EhtoE}) and (\ref{appr_mon}), we have
\begin{align}\label{Eh<E}
 E_0 := E(\chi^0) \geq E_h(\chi^0).
\end{align}
Furthermore by scaling we can normalize the prescribed volume $|\Omega^0| = \int \chi^0\,dx =1$.

\subsection{Minimizing movements interpretation}
%
%
In the following lemma we elaborate the interpretation of Algorithm \ref{MBO_volume} as a minimizing movements scheme which is the starting point
of the convergence proof.
\begin{lem}[Minimizing movements interpretation]\label{la_MM}
 Given $\chi^0\in \{0,1\}$ with $\int \chi^0\,dx =1$, let $\phi$, $\lambda$ and  $\chi^1$ be obtained by Algorithm \ref{MBO_volume}.
 Then $\chi^1$ solves
\begin{align}\label{MM_lagrange}
 \min \quad  E_h( \chi) + D_h(\chi-\chi^0) 
+ \frac{2\lambda-1}{\h}\int \chi\,dx,
\end{align}
where the minimum runs over all $\chi\colon \R^d \to \{0,1\}$.
Or equivalently
\begin{align}\label{MM_constraint}
 \min \quad  E_h( \chi) + D_h(\chi-\chi^0) \quad \textup{s.\ t.}
\quad  \int\chi \,dx= 1, \quad \chi\in\{0,1\}.
\end{align}
\end{lem}

\begin{proof}
First we show that (\ref{MM_lagrange}) is equivalent to minimizing the \lq linearized energy\rq
\begin{align}\label{def_L}
 L_{\lambda,h}(\phi,\chi) := 
\frac1\h \int \left(1-\chi\right)\phi
 + \chi \left(2\lambda -\phi\right)dx,
\end{align}
over $\chi \colon \mathbb{R}^d \rightarrow \{0,1\}$. Indeed, this is just a consequence of the fact that
\begin{align}\label{simple_computation}
E_h(\chi) + D_h(\chi - \chi^0) + \frac{2\lambda-1}{\h}\int \chi\,dx  = 
L_{\lambda, h}(\phi, \chi) + \text{Terms depending only on } \chi^0,
\end{align}

Second we show that (\ref{MM_constraint}) is equivalent to minimizing $L_{\lambda, h}(\phi,\chi)$ over 
$\chi \colon \mathbb{R}^d \rightarrow \{0,1\}$ such that $\int \chi \,dx = 1$.  This again follows from
(\ref{simple_computation}) and the fact that $\frac{2\lambda-1}{\h}\int \chi\,dx$ is a constant in this case.

Finally we show that $\chi^1$ as obtained through Algorithm \ref{MBO_volume} minimizes $L_{\lambda, h}(\phi,\chi)$ over $\chi \colon \mathbb{R}^d \rightarrow \{0,1\}$ (and therefore also minimizes $L_{\lambda, h}(\phi,\chi)$ over this class when the unit volume constraint is enforced). To see this, note that the integrand is clearly bounded below by $\phi \wedge \left(2\lambda - \phi \right)$ for any $\chi\in \{0,1\}$.
And by definition, $\chi^1$ admits this minimum pointwise:
\[
\left(1 - \chi^1\right) \phi + \chi^1 \left( 2\lambda - \phi\right) = \phi \wedge \left(2\lambda - \phi \right).\qedhere
\]
\end{proof}

%
%
The following a priori estimate is a direct consequence of the minimizing movements interpretation but is a very important tool to prove compactness of
the approximate solutions.
\begin{lem}[Energy-dissipation estimate]
The approximate solutions $\chi^h$ satisfy the following energy-dissipation estimate
\begin{align}\label{ED-estimate}
E_h(\chi^N) + \sum_{n=1}^N D_h(\chi^n-\chi^{n-1}) \leq E_0.
\end{align}
\end{lem}
\begin{proof}
 As a direct consequence of the minimization procedure (\ref{MM_constraint}) we obtain
 \[ 
 E_h(\chi^n) + D_h(\chi^n - \chi^{n-1} ) \leq E_h( \chi^{n-1} ). 
 \]
Iterating this estimate from $n=1$ to $N$ together with (\ref{Eh<E}) yields the claim.
\end{proof}

Above we used the minimizing movements interpretation to derive
an easy a priori estimate by comparing the solution $\chi^n$ to its predecessor $\chi^{n-1}$.
Now we use this interpretation to derive an optimality condition, the Euler-Lagrange
equation associated to the functional
\[
 E_h(\chi) + D_h(\chi - \chi^0) + \dfrac{2\lambda -1}{\sqrt{h}} \int \chi \,dx.
\]
This will be an important component of our convergence proof.  To state this precisely,
let us first define the notion of first variation of $E_h(\cdot)$ and $D_h(\cdot - \chi^0)$.
Since we are considering characteristic functions of sets, which induces the ``constraint'' $\chi\in\{0,1\}$, the
correct variations are \emph{inner} variations, i.\ e.\ variations of the independent variable.
Geometrically this corresponds to a deformation of the phase $\Omega$.

%
%
\begin{defn}[First variation]
 For any $\chi\in\{0,1\}$ and $\xi \in C_0^\infty(\R^d,\R^d)$ let $\chi_s$ be generated by the flow of $\xi$, i.\ e. $\chi_s$
 solves the following distributional equation:
\begin{align*}
 \partial_s \chi_s + \xi \cdot \nabla \chi_s =0.
\end{align*}
We denote the first variation along this flow by
\begin{align*}
 \delta E_h (\chi,\xi):= \frac{d}{ds} E_h(\chi_s)\big|_{s=0},\quad 
\delta D_h (\,\cdot\, - \tilde \chi)(\chi,\xi):= \frac{d}{ds} D_h(\chi_s-\tilde \chi)\big|_{s=0},
\end{align*}
where $\tilde \chi\in \{0,1\}$ is fixed.
\end{defn}

%
%
\begin{cor}[Euler-Lagrange equation]\label{cor_ELG}
 Given $\chi^0\in \{0,1\}$, let $\chi^1$ be obtained by Algorithm \ref{MBO_volume} with threshold value $\lambda$.
Then $\chi^1$ solves the Euler-Lagrange equation associated to (\ref{MM_lagrange}):
\begin{align}\label{ELG}
 \delta E_h (\chi^1,\xi)+\delta D_h (\,\cdot\, - \chi^0)(\chi^1,\xi) 
+\frac{2\lambda-1}{\h}\int\left(\nabla\cdot \xi\right) \chi^1\,dx =0.
\end{align}
\end{cor}

Equation \eqref{ELG} follows directly from the minimizing movements interpretation \eqref{MM_lagrange} and can be regarded as an approximate version of the weak formulation \eqref{H=v}.
One can easily compute the formal limit of each single term. A formal expansion suggests that with $H$ denoting the mean curvature of $\partial \Omega^1$ and $V$ denoting
the normal velocity moving $\partial \Omega^0$ to $\partial \Omega^1$ in time $h$ we have
\[
 \delta E_h (\chi^1,\xi) \approx \frac1{\sqrt{\pi}} \int_{\partial \Omega^1} H \, \xi \cdot \nu
\quad \text{and}\quad
 \delta D_h (\,\cdot\, - \chi^0)(\chi^1,\xi) \approx -\frac1{\sqrt{\pi}} \int_{\partial \Omega^1}V\,\xi \cdot \nu.
\]
Therefore, at least formally, \eqref{ELG} is similar to the desired equation $V=H-\langle H\rangle$.
In our rigorous justification we will interpret the terms in a weak sense and use the strategy
of \cite{LauOtt15}.
Following the lines of \cite{LauOtt15}, we can also compute the first variation $\delta E_h$ of the energy rigorously and obtain
\begin{align}
\delta E_h (\chi,\xi) 
= &\frac1\h \int \xi\cdot \nabla\chi \,G_h\ast \chi
 - \left(1-\chi \right)G_h \ast \left( \xi\cdot \nabla \chi\right)dx	\notag\\
= & \frac1\h \int \xi \cdot \left[ \left(1-\chi\right) \nabla G_h \ast \chi \right]
- \left(1-\chi\right) \nabla G_h \ast \left(\xi \,\chi\right) dx	\label{compute dE}\\
+& \frac1\h \int \left( \nabla\cdot \xi\right)\left(1-\chi\right)  G_h \ast  \chi
+ \left(1-\chi\right)  G_h \ast \left(\left(\nabla\cdot\xi\right) \chi\right)dx. \notag
\end{align}
Expanding $\xi(x)-\xi(x-z) = \left(z\cdot \nabla\right) \xi(x) +O(|z|^2)$ for the first right-hand side integral,
and $\left(\nabla\cdot \xi\right) (x-z) = \left(\nabla\cdot \xi\right) (x) +O(|z|)$ for the second we obtain
\begin{equation}\label{compute dE 2}
  \delta E_h(\chi,\xi) = \frac1\h \int \nabla \xi \colon \left(1-\chi\right) \left( G_h\, Id - 2h \nabla G_h\right) \ast \chi\, dx + o(1),
 \end{equation}
 as $h\to0$.   
The integral on the right hand side formally converges to $\frac{1}{\sqrt{\pi}} \int \nabla \xi \colon \left (Id - \nu\otimes\nu \right) \left | \nabla \chi \right |$, and can be made rigorous.
 We will discuss this fact below in Proposition \ref{deltaE}.
For the first variation of the dissipation we can expand $\xi$ again and obtain
\[
 \delta D_h (\,\cdot\, - \chi^0)(\chi^1,\xi) = -2\int \frac{\chi^1-\chi^0}{h} \xi \cdot \h \nabla G_h\ast \chi^1\,dx +o(1),
\]
where the firstfactor in the right-hand side integral is a finite difference and formally converges to $\partial_t \chi = V\,|\nabla\chi|$, and the second factor formally converges to
$\frac{1}{2\sqrt{\pi}} \nu$.
The rigorous justification of this fact is more involved since one has to pass to the limit in a product of two weakly converging terms.
We will show how to overcome this difficulty in the following.

\subsection{Main result}
From (\ref{ELG}) we establish convergence to the weak formulation of volume-preserving mean-
curvature flow in Definition \ref{def_motion_by_mean_curvature}.
The central novelties of this section are establishing the equivalence of
(\ref{ELG}) to Algorithm \ref{MBO_volume}, which was done above, and to show that the
threshold value $\lambda$ remains close to $\frac{1}{2}$ in a certain sense, which is done
in Prop. \ref{L2 bounds lambda} below. The latter property plays an important role
in showing that each of the three terms of (\ref{ELG}) converges to its respective limit.
The mean curvature is recovered as the limit of the first variation $\delta E_h$ of the
energies (c.f. Prop. \ref{deltaE}), and the normal velocity is recovered as the limit of the
first variation $\delta D_h$ of the dissipation (c.f. Prop. \ref{deltaD}).
Doing so is similar to results in
\cite{LauOtt15}, however technical difficulties must be overcome due to the fact that the threshold
parameter $\lambda$ may vary (as opposed to being fixed at $\frac{1}{2}$ in the original MBO scheme).
The averaged mean curvature is recovered as the limit of the Lagrange multipliers, c.f. proof of Thm.\ref{thm1}.

We now state and prove the main result of this section, Theorem \ref{thm1} below.
Under the same convergence assumption as in \cite{LauOtt15} which is inspired by the assumption in \cite{LucStu95} we can prove the convergence of the scheme.
 For clarity of presentation, the given proof merely highlights the main ideas involved in establishing the convergence of (\ref{ELG}) to 
(\ref{H=v}).  The more technical aspects of the proof are then
postponed to later subsections (c.f. Props. 
\ref{L2 bounds lambda}, \ref{compactness}, \ref{tightness}, \ref{deltaE}, \ref{deltaD}).
\begin{thm}\label{thm1}
Let $T<\infty$ and $\chi^0\in\{0,1\}$ with $E(\chi^0) <\infty$ and $\{\chi^0=1\} \subset \subset \R^d$.
After passage to a subsequence, the functions $\chi^h$ obtained by Algorithm \ref{MBO_volume}
converge to a function $\chi$ in $L^1((0,T)\times \R^d)$. Under the convergence assumption
(\ref{conv_ass}), $\chi$ is a solution of the volume-preserving mean-curvature flow equation
in the sense of Definition \ref{def_motion_by_mean_curvature}.
\end{thm}
\begin{proof}[Proof of Theorem \ref{thm1}]
 By Proposition \ref{compactness} the approximate solutions $\chi^h$ converge to some limit $\chi$ after passage to a subsequence.
 The strategy of our proof for (\ref{H=v}) is to pass to the limit in the Euler-Lagrange equation (\ref{ELG}) after integration in time.
 
 By Proposition \ref{L2 bounds lambda}, after passing to a further subsequence, we can find a function $\Lambda \in L^2(0,T)$ such that 
\begin{align*}
 \frac{2\lambda_h-1}{\h} \rightharpoonup \frac{1}{\sqrt{\pi}}\Lambda\quad \text{in }L^2(0,T).
\end{align*}
Since the integrals converge strongly,
\begin{align*}
 \int\left(\nabla\cdot \xi\right) \chi^h \, dx
 \to \int \left(\nabla\cdot \xi \right) \chi \,dx \quad \text{in }L^2(0,T),
\end{align*}
we can pass to the limit $h\to 0$ in the product.
This is one of the three terms of the Euler-Lagrange equation.
In Proposition \ref{deltaE} we recover the mean curvature from the first variation of the energy,
i.\ e. the first term in (\ref{ELG}). In Proposition \ref{deltaD}
we recover the normal velocity from the second term in (\ref{ELG}), the first variation of the
dissipation.
Therefore, the limit solves (\ref{H=v}).
Furthermore, $V$ solves (\ref{v=dtX}) by construction.
Note that since $\Lambda,\,V\in L^2(\left|\nabla \chi \right| dt)$ we have a generalized mean curvature 
$H\in L^2(\left|\nabla \chi \right| dt)$.
We are left with proving (\ref{average mean curvature}).
Note that $t\mapsto \int \chi(t)\,dx \in H^1(0,T)$ with
\begin{align*}
 \frac d{dt}\int \chi\, dx = \int V \left| \nabla \chi \right|.
\end{align*}
Indeed, given $f\in C_0^\infty(0,T)$ and $g\in C_0^\infty(\R^d)$ with
$g \equiv 1$ on $B_{R^\ast}$ with $R^\ast= R^\ast(d,E_0,T)$ from Proposition \ref{tightness}, setting $\zeta(x,t) := f(t)g(x)$, 
we have
\begin{align*}
-\int_0^T  f'(t) \int \chi(t) \,dx \,dt
=  -\int_0^T \int \partial_t \zeta \, \chi \, dx\,dt
=  \int_0^T \int \zeta \,V \left| \nabla \chi \right| dt
=  \int_0^T f(t) \int V \left| \nabla \chi \right| dt.
\end{align*}
Since $\int \chi^h \,dx $ is constant in time,
also $\int \chi\, dx$ is constant in time.
Using (\ref{H=v}) as a pointwise a.\ e. statement in time, we have
\begin{align*}
 0 = \frac d{dt} \int \chi\, dx = \int V \left| \nabla \chi \right|
 \overset{(\ref{H=v})}{=} \int \left(H-\Lambda\right) \left| \nabla \chi \right|
 =   \int H \left| \nabla \chi \right|-\Lambda \int \left| \nabla \chi \right|
\end{align*}
almost everywhere in $(0,T)$. Solving for $\Lambda$ yields (\ref{average mean curvature}).
\end{proof}

\subsection{\texorpdfstring{$L^2$}{L2}-estimate for Lagrange multipliers}
%
%
The following proposition gives a quantitative estimate on the closeness of the threshold values $\lambda_n$ to $\tfrac12$ in the natural topology coming 
from the gradient flow structure and the appearance of $\tfrac{2\lambda_n-1}\h$ as a Lagrange multiplier. 
Roughly speaking, the lemma states that $\left|\lambda_h -\tfrac12\right| = \mathcal{O}(\h)$ in $L^2$.
This is the analogue of Corollary 3.4.4 in \cite{mugnai2015global} but our proof works in a different way.
While they couple the bound on the Lagrange multiplier and the growth rate of the sets via the estimate (3.28) in \cite{mugnai2015global},
we prove the bound on the Lagrange multipliers first, independently of the growth rate.
The main difference is that we construct our test function $\xi$ via some elliptic problem in Step 3 of the proof below so that we can obtain
estimates by using elliptic regularity theory, in particular the Calder\'{o}n-Zygmund inequality, cf. Theorem 9.9 in \cite{gilbarg2001}.	

\begin{prop}[\texorpdfstring{$L^2$}{L2}-estimate for Lagrange multipliers]\label{L2 bounds lambda}
 Given the approximate solutions $\chi^h$ obtained by Algorithm
\ref{MBO_volume} with threshold values $\lambda_h$, for $h\ll \frac{1}{E_0^2}$ we have
\begin{align*}
 \int_0^T \left( \lambda_h - \tfrac12\right)^2 dt \lesssim \left(1+T\right)\left( 1+E_0^4\right) h.
\end{align*}
Here $h\ll \frac1{E_0^2}$ means that there exists a generic constant $C=C(d)<\infty$ such that the statement holds for $h< \frac1{C E_0^2}$.
We recall that $A\lesssim  B$ means $A\leq C\,B$ for some generic constant $C=C(d)<\infty$.
\end{prop}
\begin{proof}
Squaring the Euler-Lagrange equation (\ref{ELG}), we obtain
\begin{align}\label{ELGstepn}
 \frac1h\left( \lambda_n - \tfrac12\right)^2
\left( \int\left(\nabla\cdot \xi\right) \chi^n\,dx\right)^2
\lesssim \left[\delta E_h (\chi^n,\xi)\right]^2
+\left[\delta D_h (\,\cdot\, - \chi^{n-1})(\chi^n,\xi)\right]^2
\end{align}
for any $\xi\in C_0^\infty(\R^d,\R^d)$.
In order to prove the proposition, we first estimate the right-hand side for an arbitrary test vector field $\xi$,
cf. Step 1 for the first and Step 2 for the second term.
In Step 3 we construct a specific vector field such that the integral on the left-hand side is bounded from below.
\step{Step 1: Estimates on $\delta E_h (\chi,\xi)$.}
For any $\chi\in\{0,1\}$ and any $\xi\in C_0^\infty(\R^d,\R^d)$, we have
\begin{align}\label{estimate_dE}
 \left|\delta E_h (\chi,\xi)\right|
\lesssim \left\| \nabla \xi\right\|_\infty E_h(\chi).
\end{align}
Argument:
Starting from the computation \eqref{compute dE} we see that the second integral on the right-hand side is clearly controlled by
$
 \|\nabla\xi\|_\infty E_h(\chi),
$
whereas the first integral on the right-hand side can be estimated via
\begin{align*}
 &\frac1\h \int \xi \cdot \left[ \left(1-\chi\right) \nabla G_h \ast \chi \right]
- \left(1-\chi\right) \nabla G_h \ast \left(\xi \chi\right) dx\\
&= \frac1\h\int -\frac{ z}{2h} G_h(z) \cdot \int  \left( \xi(x) - \xi(x-z) \right)
\left( 1-\chi \right)(x) \chi(x-z) \,dx\,dz\\
&\leq \left\| \nabla \xi\right\|_\infty \frac1\h \int  \frac{ |z|^2}{2h} G_h(z)
 \int  \left( 1-\chi \right)(x) \chi(x-z) \,dx\,dz.
\end{align*}
Using $|z|^2G_1(z) \lesssim G_{2}(z)$ we thus have
\begin{align*}
 \left|\delta E_h (\chi,\xi) \right|
\lesssim \left\| \nabla \xi\right\|_\infty \left(E_{2h}(\chi) +  E_{h}(\chi)\right)
\end{align*}
and the approximate monotonicity of the energy (\ref{appr_mon}) yields (\ref{estimate_dE}).
\step{Step 2: Estimates on $\delta D_h (\,\cdot\, - \chi^{n-1})(\chi^n,\xi)$.}
We have
\begin{align}\label{estimate_dD}
  h\sum_{n=1}^N\left[\delta D_h (\,\cdot\, - \chi^{n-1})(\chi^n,\xi_n) \right]^2
\lesssim \sup_n \left\| \xi_n\right\|_{W^{1,\infty}}^2\left( 1+ E_0^2\right).
\end{align}
Argument: For any $\xi\in C_0^\infty(\R^d,\R^d)$ and any $n\in \{1,\dots,N\}$, we have
\begin{align*}
 \delta D_h (\,\cdot\, - \chi^{n-1})(\chi^n,\xi) 
=& \frac2\h \int \left( -\xi \cdot \nabla \chi^n\right)G_h\ast \left(\chi^n-\chi^{n-1}\right) dx\\
= &\frac2\h \int \chi^n \xi  \cdot \nabla G_h\ast \left(\chi^n-\chi^{n-1}\right)
+ \left( \nabla\cdot \xi\right)\chi^n G_h\ast \left(\chi^n-\chi^{n-1}\right) dx.
\end{align*}
Setting (compare to the \emph{dissipation measures} $\mu_h$ in Definition 2.7 in \cite{LauOtt15})
\begin{align*}
 \mu_n := \frac1\h \int \left[G_{h/2}\ast\left(\chi^n-\chi^{n-1} \right) \right]^2 dx
\end{align*}
and using the Cauchy-Schwarz inequality, we obtain
\begin{align*}
 \left[\delta D_h (\,\cdot\, - \chi^{n-1})(\chi^n,\xi) \right]^2
&\lesssim 
\left(\frac1h \int \h \nabla G_{h/2}\ast\left(\chi^n \xi\right)
  G_{h/2}\ast \left(\chi^n-\chi^{n-1}\right) dx\right)^2\\
&+ \left\| \nabla \xi\right\|_\infty ^2
\left(\frac1\h \int G_{h/2}\ast\chi^n\left|G_{h/2}\ast \left( \chi^n-\chi^{n-1}\right) \right|dx\right)^2\\
&\leq 
\frac1h\left(\frac1\h \int \left[\h \nabla G_{h/2}\ast\left(\chi^n \xi\right)\right]^2 \!dx\right)
  \mu_n
+ \frac1{\h} \left\| \nabla \xi\right\|_\infty^2
\int \chi^n\,dx\;\mu_n.
\end{align*}
For the first right-hand side term, we first observe that for any $\chi\in \{0,1\}$ and any $\xi\in C_0^\infty(\R^d,\R^d)$,
by $|\xi(x+z)-\xi(x)| \leq \|\nabla \xi\|_\infty |z|$ we obtain
\begin{align*}
 &\frac1\h\int |\h \nabla  G_{h/2}(z)| \int \big|\xi(x+z)-\xi(x)\big| \chi(x+z)
\left| \h \nabla G_{h/2}\ast \left( \chi \xi \right)\right|(x) \,dx\,dz\\
&\leq \left\| \xi\right\|_\infty\left\| \nabla \xi\right\|_\infty \int \chi\,dx \left(\int |z||\nabla G_{h/2}(z)|\,dz\right) \left( \int |\h \nabla  G_{h/2}(z)|\,dz\right),
\end{align*}
where the last two integrals are uniformly bounded in $h$.
Thus, in our case where $\chi=\chi^n$ with $\int \chi^n\,dx =1 $, we obtain an estimate on the error when commuting the multiplication with $\xi$ and the convolution with
the kernel $\h \nabla G_{h/2}$ in one of the factors:
\begin{align*}
 \frac1\h \int \left[\h \nabla G_{h/2}\ast\left(\chi \xi\right)\right]^2 \!dx
\leq \frac1\h \int \xi \!\cdot \h \nabla G_{h/2}\ast\chi
 \left[\h \nabla G_{h/2}\ast\left(\chi \xi\right)\right] \!dx 
+ c(d) \left\| \xi\right\|_{W^{1,\infty}}^2.
\end{align*}
Since $\nabla G$ is antisymmetric and since $\left|z\right|G(z) \lesssim G_2(z)$, we have
\begin{align*}
 &\frac1\h \int \xi \cdot \h \nabla G_{h/2}\ast\chi
 \left[\h \nabla G_{h/2}\ast\left(\chi \xi\right)\right] \!dx\\
&\qquad \qquad \qquad \qquad = \frac1\h \int \xi \cdot \h \nabla G_{h/2}\ast\left(\chi-1\right)
 \left[\h \nabla G_{h/2}\ast\left(\chi \xi\right)\right] \!dx \\
&\qquad \qquad \qquad \qquad \lesssim \left\| \xi\right\|_\infty^2 \frac1\h \int G_{h}\ast \left(1-\chi\right) 
G_{h}\ast\chi \,dx
\lesssim \left\| \xi\right\|_\infty^2 E_h(\chi).
\end{align*}
Thus, we have
\begin{align*}
  \left[\delta D_h (\,\cdot\, - \chi^{n-1})(\chi^n,\xi) \right]^2
\lesssim& \frac1h\left( 
\left\| \xi\right\|_\infty^2 E_0 + \left\| \xi\right\|_{W^{1,\infty}}^2 +
 \h \left\| \nabla \xi\right\|_\infty^2\right) \mu_n,
\end{align*}
which is (\ref{estimate_dD}) after integration in time and using the energy-dissipation
estimate (\ref{ED-estimate}) once more.
\step{Step 3: Choice of $\xi$.}
For any $E_0>0$, any $0< h \ll 1/E_0^2$ 
and any $\chi\in \{0,1\}$ with $\int \chi\,dx =1$, $\supp \chi \subset \subset \R^d$
and $E_h(\chi)\leq E_0$ there exists $\xi\in C_0^\infty(\R^d,\R^d)$ with
\begin{align}
 \int \left(\nabla\cdot \xi\right) \chi \,dx &\geq \frac 12\quad \text{and} \label{divxi=chie}\\
\left\|  \xi\right\|_{W^{1,\infty}}		&\lesssim 1+ E_0.
\label{boundDxi}
\end{align}
Argument:
Set $\varepsilon^2=\frac1{ C E_0^2}$. We will determine the constant $C=C(d)$ later. 
Set $\chi_\varepsilon :=  \varphi_\varepsilon\ast \chi$ for some standard
mollifier $\varphi_\varepsilon(z)=\frac1{\varepsilon^d}\varphi_1(\frac z\epsilon)$ with $0\leq \varphi_1\leq 1$, $\int \varphi_1 \,dz =1$, $\varphi_1 \lesssim G_1$ and $\int |\nabla \varphi_1| \,dz \lesssim 1.$ 
Then $\chi_\varepsilon\in C_0^\infty(\R^d,[0,1])$.
Let $u$ denote the solution of 
\begin{align*}
 \Delta u = & \chi_\varepsilon
\end{align*}
given by the Newtonian potential
$
 u = \Gamma \ast \chi_\varepsilon.
$
We define $ \xi:=  \nabla u = \Gamma \ast \nabla \chi_\varepsilon$
and claim that $\xi$ satisfies (\ref{divxi=chie}). 
Indeed, since 
$\left| \chi_\varepsilon - \chi \right| = \chi\left( 1-\chi_\varepsilon \right) + \left( 1-\chi \right)\chi_\varepsilon$ for $\chi\in\{0,1\}$ and $0\leq \chi_\varepsilon \leq 1$,
we can use the approximate monotonicity (\ref{appr_mon}) such that for any $0<h\leq \varepsilon^2$ we have
\begin{align*}
 \int\!\left| \chi_\varepsilon - \chi \right| dx 
&= 2 \int \!\left( 1-\chi\right) \varphi_\varepsilon \ast \chi \,dx
\lesssim \int  \!\left( 1-\chi\right) G_{\varepsilon^2} \ast \chi \,dx \!
\overset{(\ref{appr_mon})}{\leq} \!
 \varepsilon \left(\frac{\varepsilon + \h }{\varepsilon} \right)^{d+1}\!\!E_h(\chi)
\lesssim \varepsilon E_0.
\end{align*}
Thus, if we pick the constant $C(d)$ in the definition of $\varepsilon$ large enough, we have
\begin{align*}
 \int \left(\nabla\cdot\xi\right)\chi\,dx =  \int\chi_\varepsilon\chi\,dx
\geq \int \chi \,dx -  \int\left| \chi_\varepsilon - \chi \right| dx 
\geq \frac 12,
\end{align*}
which is (\ref{divxi=chie}).
Now we give an argument for (\ref{boundDxi}).
The Calder\'{o}n-Zygmund inequality yields
\begin{align}\label{CZforDxi}
\int_{\R^d} | \nabla \xi |^p dx 
 \lesssim_p  \int |\chi_\varepsilon |^p dx \leq 1
\end{align}
for any $1<p<\infty$, where we write $\lesssim_p $ to stress that the constant depends not only on the dimension $d$ but also on the parameter $p$.
Since $\chi_\varepsilon$ is smooth, we can differentiate the equation:
\begin{align*}
 \Delta \xi = \nabla \chi_\varepsilon.
\end{align*}
Thus by the Calder\'{o}n-Zygmund inequality and Jensen's inequality
\begin{align}\label{CZforD2xi}
\int_{\R^d} | \nabla^2 \xi |^p dx 
\lesssim_p \int |\nabla \chi_\varepsilon |^p dx \leq\left(\int \left| \nabla \varphi_\varepsilon\right| dz \right)^p \int \left|\chi\right|^pdx \lesssim\frac1{\varepsilon^p}
\end{align}
for any $1<p<\infty$.
Now we want to bound the $0$-th order term of $\xi$.
Let $R>0$ be big enough such that $\supp \chi_\varepsilon \subset B_{\frac R2}$ 
and take $\eta \in C^{\infty}_c(B_{2R})$ to be a cut-off function for $B_{R}$ in $B_{2R}$ with $|\nabla\eta|\lesssim \frac1R$.
Then we have
\begin{align*}
\int |\nabla (\eta\,\xi)|^pdx 
&\lesssim_p  \int \eta | \nabla  \xi|^pdx 
+\int |\nabla \eta |^p | \xi|^pdx 
\overset{(\ref{CZforDxi})}{\lesssim}\!\!\!_p   \;1
+ \frac1{R^p} \int_{B_{2R}\setminus B_{R}}  | \xi|^pdx.
\end{align*}
Note that for any $x\in \R^d\setminus B_{R}$, since then $\dist(x,\supp \chi_\varepsilon) \gtrsim R$, we have
\begin{align*}
 | \xi(x)|\leq \int \left| \nabla \Gamma(x-y)\right| \chi_\varepsilon(y) \, dy
\lesssim \frac1{R^{d-1}} \int \chi_\varepsilon(y)\,dy = \frac1{R^{d-1}} .
\end{align*}
Thus,
\begin{align}\label{Detaxi}
\int |\nabla(\eta\, \xi)|^p dx \lesssim_p 1+  R^{d(1-p)}.
\end{align}
Now we fix some $p = p(d) \in (\frac d 2,d)$.
Since $\eta\,\xi$ has compact support, we can apply the Gagliardo-Nirenberg-Sobolev inequality, so that
\begin{align*}
\int_{B_R} |\xi|^{p^\ast}dx
\leq\int |\eta \,\xi|^{p^\ast}dx
\lesssim \left(\int |\nabla(\eta\,\xi)|^{p}dx\right)^{p^\ast/p}
\overset{(\ref{Detaxi})}{\lesssim} \left( 1+  R^{d(1-p)}\right)^{d/(d-p)},
\end{align*}
where $p^\ast = \frac{pd}{d-p}>d$ is the Sobolev conjugate of $p$.
Taking the limit $R\to \infty$, we obtain
\begin{align}\label{Lpforxi}
\int |\xi|^{p^\ast}dx \lesssim 1.
\end{align}
Since $p^\ast>d$, by Morrey's inequality and the above estimates (\ref{CZforDxi}), (\ref{CZforD2xi}) with $p^\ast$ playing the role of $p$ and (\ref{Lpforxi}), we have
\begin{align*}
\| \xi\|_{W^{1,\infty}(\R^d)}
\lesssim \|  \xi\|_{W^{2,p^\ast}(\R^d)} \lesssim 1+\frac1{\varepsilon} \sim 1+E_0.
\end{align*}
\step{Step 4: Conclusion.} We apply Step 3 on $\chi=\chi^n$ and find $\xi^n\in C_0^\infty(\R^d,\R^d)$
with
\begin{align*}
 \int \left(\nabla\cdot \xi^n\right) \chi^n \,dx &\geq \frac 12\\
\left\|  \xi^n\right\|_{W^{1,\infty}}		&\lesssim 1 +  E_0.
\end{align*}
Plugging $\xi=\xi^n$ into (\ref{ELGstepn}), summing over $n$ and using the estimates in Steps 1 and 2, we obtain
\begin{align*}
 \sum_{n=1}^N \left( \lambda_n-\tfrac12\right)^2
\lesssim &\sup_n\left\| \xi^n\right\|_{W^{1,\infty}}^2
\left(T E_0^2 + 1+  E_0^2\right)
\lesssim  (1+T)(1+E_0^4),
\end{align*}
which is the desired estimate.
\end{proof}
\subsection{Compactness}
%
%
\begin{prop}[Compactness]\label{compactness}
 There exists a subsequence $h\searrow 0$ and a function 
$\chi\in L^1((0,T)\times\R^d,\{0,1\})$ such that
\begin{align}\label{comp_conv_1}
 \chi^h\longrightarrow \chi \quad \text{in } L^1( (0,T)\times\R^d).
\end{align}
Moreover,
\begin{align}\label{comp_conv_2}
 \chi^h\longrightarrow \chi \quad \text{a.\ e. in } (0,T)\times \R^d
\end{align}
and $\chi(t)\in BV(\R^d,\{0,1\})$, $\int\chi(t)\,dx =1 $ for a.\ e. $t\in(0,T)$.
\end{prop}
\begin{proof}
As in Lemmas 2.4 and 2.5 in \cite{LauOtt15} we can prove that
\begin{align}\label{compactness dtchi}
 \int_0^T \int \left| \chi^h(x+\delta \,e, t+\tau) - \chi^h(x,t)\right|dx\,dt \lesssim \left(1+T\right)E_0 \left(\delta+\tau +\h\right).
\end{align}
The proposition follows then from the arguments in Proposition 2.1 of \cite{LauOtt15}
in conjunction with Proposition \ref{tightness} below.
Indeed, in \cite{LauOtt15}, the authors show that this can be done by adapting the proof of the Riesz-Kolmogorov compactness theorem.
Since we work in $\R^d$ and not on a periodic domain as in \cite{LauOtt15} we need to guarantee that no mass escapes to infinity.
The proposition below establishes precisely this.
\end{proof}

Take $R_0 > 0$ such that $\Omega^0 \subset B_{R_0}$.
For subsequent $n$ we take a sequence of radii $R_n \geq R_{n-1}$ such that 
$\Omega^n \subset B_{R_n}$.
The focus of this section will be to show that we can choose the radii $R_n$ such that they are uniformly bounded for 
$n\in \{1,\dots,N\}$, independent of the time step $h$.
\begin{prop}[Tightness]\label{tightness}
 There is a finite radius $R^\ast=R^\ast(d,E_0,T)$, independent of $h$ such that 
\begin{align*}
 \Omega^h(t) \subset B_{R^*} \quad \text{for all } t\in[0,T].
\end{align*}
\end{prop}

We seperate the indices $n$ into `good' and `bad' iterations.  
A `good' iteration is taken to mean that $|\lambda_n - \frac{1}{2}| < \frac{1}{4}$,
and a bad iteration will be taken to mean that $|\lambda_n - \frac{1}{2}| \geq \frac{1}{4}$. 
The $L^2$-bounds in Proposition \ref{L2 bounds lambda}
give us a suitable level of control over the number of `bad' iterations. 
Indeed, Chebyshev's inequality implies that the number of `bad'
iterations is controlled by $(1+T)(1+E_0^4).$

In the next Lemma we show that in the worst case scenario, 
the radii $R_n$ grow exponentially over consecutive iterations.
\begin{lem}\label{lem_bad_iteration}
$R_n$ may be chosen such that $R_n \leq 3 R_{n-1}$.
\end{lem}
\begin{proof}
In order to reduce the notation we may assume $n=1$ and write $\phi = G_h \ast \chi^{0}$, $R:=R_0$,
$\chi = \chi^1$ and $\lambda = \lambda_1$. 
We first claim that
\begin{align}\label{estimate on phi}
 \min_{ \overline B_{R}} \phi > \max_{\R^d\setminus B_{3R}} \phi.
\end{align}
This follows immediately from the definition of $\phi$ using $\{\chi^0=1\}\subset B_R$ and the obvious inequality
\begin{align*}
 \left| x- z \right| <  2R < \left| y-z \right| 
\quad \text{for all } x\in B_R, \, y\in \R^d\setminus B_{3R}\text{ and }z\in B_R.
\end{align*}


Now suppose that $U := \Omega \setminus  B_{3R}$ has positive measure.  
This being the case, we may construct a new set, call it $\widetilde{\Omega}$,
by deleting the volume $U$ from $\Omega \setminus  B_{3R}$ and filling it 
into $B_{R}$.
Indeed, since $\left| \Omega\right| = \left| \Omega^0\right|$, we can find a set $\widetilde U \subset B_R$ of the same volume as $U$ 
such that $\widetilde U \cap \Omega = \emptyset$.
Then we set $\widetilde \Omega := (\Omega \setminus U) \cup \widetilde U$ and $\tilde{\chi} = \chara_{\widetilde{\Omega}}$.  
Recall the definition of $L_h$ in (\ref{def_L}).
We claim that $\tilde{\chi}$ has lower linearized energy $L_h(\phi, \cdot)$ than $\chi$,
which is a contradiction.  
By $\int \tilde \chi \,dx = \int \chi\,dx$ and (\ref{estimate on phi}) we have
\begin{align*}
L_h(\phi,\chi) - L_h(\phi,\tilde \chi) 
 = \frac{2}{\sqrt{h}} \int \phi \left( \tilde \chi - \chi \right)  dx 
 = \frac{2}{\sqrt{h}} \int \phi \left( \chara_{\tilde U} - \chara_{U}  \right)dx > 0.
\end{align*}
Thus we conclude that the minimizer of the linearized energy $L_h(\phi, \cdot)$ 
cannot contain any volume outside $B_{3R}$.
\end{proof}

Next we show that over `good' iterations, i.\ e. $ | \lambda_n-\tfrac12 | < \tfrac14$, the growth of $R_{n-1}$ to $R_n$ is 
$\mathcal{O}(|\lambda_n-\tfrac12|\sqrt{h})$, which in terms of Proposition 
\ref{L2 bounds lambda} can be interpreted as `linear growth'.

\begin{lem}\label{lem_good_iteration}
 There exists a universal constant $C<\infty$ such that over `good' iterations we have
\begin{align*}
 R_n \leq R_{n-1} + C \sqrt{h}| \lambda_n -\tfrac12|.
\end{align*}
\end{lem}
\begin{proof}
Given  $|\lambda_n-\tfrac12| < \frac{1}{4}$, 
we want to find a constant $C<\infty$ so that for any direction $e\in \sphere$
we have $\phi < \lambda_n$ and therefore $\chi^{n} = 0$ in $\{x\cdot e > R_{n-1} + C \h| \lambda_n-\tfrac12 |\}.$ 
We prove this by comparing to a half space  $H = \{x\cdot e < R_{n-1}\}$
 whose boundary is tangent to $\partial B_{R_{n-1}}$.
By rotational symmetry we may assume w.\ l.\ o.\ g. that $e=e_1$ so that at a point $x=(x_1,x')$, thanks to the
factorization property of $G$, we can estimate
$$\phi = G_h \ast \chi^{n-1} \leq G_h \ast \chara_H
= \int_{-\infty}^\infty G^1_h(z_1) \chara_{x_1 + z_1 < R_{n-1}} dz_1
= \frac{1}{2} - \int_{0}^{x_1 - R_{n-1}} G^1_h(z_1)\, dz_1. $$
We observe that the right-hand side expression is monotone decreasing in $x_1$ and
find the upper bound for $R_n\geq R_{n-1}$ simply by setting the right-hand side to be equal to 
$\lambda_n$ for $x_1=R_n$:
\begin{align*}
| \lambda_n - \frac{1}{2}|  = \int_0^{\frac1\h({R_n - R_{n-1}})} G^1(z_1) \,dz_1.
\end{align*}
There exists a universal $C<\infty$ such that $\int_0^{C}G^1(z_1) \,dz_1 = \frac{1}{4}$.  Thus,
since $| \lambda_n - \frac{1}{2} | < \frac{1}{4}$, we have $\frac{R_n - R_{n-1}}{\sqrt{h}} <C$.
In turn this gives
$$ \frac{R_n - R_{n-1}}{\sqrt{h}} \min_{|z_1| \leq C } G^1(z_1) < | \lambda_n - \tfrac12 |,$$
which is the desired estimate.
\end{proof}
\begin{proof}[Proof of Proposition \ref{tightness}]
 The result follows by iterating the estimate of the previous two lemmas.  Indeed, over `good' iterations
we have the estimate
$$ R_n \leq R_{n-1} +  C \sqrt{h} | \lambda_n - \tfrac12 |.$$
And over `bad' iterations we have the estimate
$$ R_n \leq 3 R_{n-1}. $$ 
Iterating these two estimates and keeping in mind that we have at most a finite number $\sim (1+T)(1+E_0^4)$
of `bad' iterations we obtain
$$
R_N \leq C(d,T,E_0) \left(R_0 + \sum_{n = 1}^N \sqrt{h} | \lambda_n - \tfrac12 |\right).
$$
Finally we note that by Jensen's inequality and Proposition \ref{L2 bounds lambda}
\begin{align*}
 \sum_{n = 1}^N \sqrt{h} | \lambda_n - \tfrac12 |  
 \leq \left( h \sum_{n=1}^N \frac{|\lambda_n - \tfrac12 |^2}{h} \right)^{\frac{1}{2}} T^{\frac{1}{2}} 
 \leq C(d, E_0, T).
\end{align*}
The constant $C(d,E_0,T)$ yields the estimate on $R^\ast$.
Note that our proof does not give a linear growth estimate in time. 
Indeed, the upper bound $R^\ast$ growth exponentially in $T$. Nevertheless, for our purpose, this is enough. 
\end{proof}

\subsection{Convergence}
%
%
In this section we give the details of the proof of Theorem \ref{thm1}.
We can directly apply Proposition 3.1 of \cite{LauOtt15} to our situation, which we state in Proposition \ref{deltaE}.
In Proposition \ref{deltaD} we prove that we can change the proof of Proposition 4.1 of \cite{LauOtt15} so that it applies in our situation.
For this part we need Proposition \ref{L2 bounds lambda} to apply the one-dimensional lemma, Lemma \ref{1d lemma} stated below.

\begin{prop}[Energy and mean curvature; Prop.\ 3.1 in \cite{LauOtt15}]\label{deltaE}
 Under the convergence assumption (\ref{conv_ass}) we have
 \begin{align*}
 \lim_{h\to0}\int_0^T \delta E_h(\chi^h,\xi)\,dt  =  \frac1{\sqrt{\pi}}\int_0^T \int \left( \nabla\cdot \xi - \nu\cdot \nabla \xi \, \nu\right) \left|\nabla\chi\right|  dt
 \end{align*}
 for any $\xi \in C_0^\infty((0,T)\times \R^d,\R^d).$
\end{prop}
\begin{proof}
 The proof of Proposition 3.1 in \cite{LauOtt15} only uses the convergence that we deduced here in Proposition \ref{compactness}
 and the convergence assumption.
 However, we briefly highlight the line of proof here.
 We observe that the expansion \eqref{compute dE 2} of the first variation of the energy is already in the same form as the limit: multiplication with the anisotropic kernel $G_h\, Id- 2h\nabla G_h $ corresponds to
 multiplication with $Id-\nu\otimes\nu$, i.e.\ projection onto the tangent space.
 More precisely, evaluated at a fixed configuration $\chi$, the right-hand side of \eqref{compute dE 2} converges to the correct quantity.
 Under the strengthened convergence \eqref{conv_ass} this holds true also along the sequence $\chi^h$.
\end{proof}

\begin{prop}[Dissipation and normal velocity]\label{deltaD}
 There exists a function $V\colon (0,T)\times \R^d \to \R$ which is a normal velocity in the sense of (\ref{v=dtX}).
 Given the convergence assumption (\ref{conv_ass}), $V\in L^2(\left| \nabla\chi\right|dt)$ and for any $\xi \in C_0^\infty((0,T)\times \R^d,\R^d)$ we have
 \begin{align}\label{eq_deltaD}
 \lim_{h\to0}\int_0^T \delta D_h(\,\cdot\,, \chi^h(t-h))(\chi^h(t),\xi(t))\,dt  
 =  - \frac1{\sqrt{\pi}}\int_0^T \int V\, \xi \cdot \nu \left|\nabla\chi\right|  dt.
 \end{align}
\end{prop}
\begin{proof}
 Since we have the same energy-dissipation estimate, namely (\ref{ED-estimate}), with the volume constraint as in \cite{LauOtt15} without a constraint,
 we can directly apply most of the techniques.
 In Lemma \ref{1d lemma}, we show that for most of the iterations we can also apply the finer estimate, Lemma 4.2 in \cite{LauOtt15} when changing 
 the threshold value from $\tfrac12$ to $\lambda$ as in Step 2 of Algorithm \ref{MBO_volume}.
 To make this applicable we need the $L^2$-estimate in Proposition \ref{L2 bounds lambda}.
 \step{Step 1: Construction of the normal velocity and (\ref{v=dtX}).}
  We construct the normal velocity $V$ exactly as in Lemma 2.11 in \cite{LauOtt15}.
  First, one proves that the distributional time derivative $\partial_t \chi$ of $\chi$ is a Radon measure using only
  the energy-dissipation estimate, in our case (\ref{ED-estimate}).
  Using the convergence assumption, for us (\ref{conv_ass}), this measure turns
  out to be absolutely continuous w.\ r.\ t. $\left|\nabla \chi\right| dt$,
  so that one can define $V$ to be the density of $\partial_t \chi$ w.\ r.\ t. $\left|\nabla \chi\right| dt$
  and prove higher integrability, $V\in L^2(\left|\nabla \chi\right| dt)$.
  Then $V$ satisfies (\ref{v=dtX}) by construction.
 \step{Step 2: Argument for (\ref{eq_deltaD}).}
  One of the key ideas in \cite{LauOtt15} is to introduce a mesoscopic time scale $\alpha\h$.
  In Step 2 of the proof of Proposition 4.1 there, one chooses a shift of the mesoscopic time slices
  so that one has control over the error terms. 
  We can make use of this degree of freedom to make sure that in addition 
  the mesoscopic time steps are `good' iterations.
  Given $N=T/h$, $\kmax = \alpha/\h$, 
  $\lmax=N/\kmax$, for any function $\gap \colon \{1,\dots,N\}\to [0,\infty)$
  we can find $k_0\in \{1,\dots,\kmax\}$, such that in addition to
  \begin{align}
    \frac1\lmax \sum_{l=1}^\lmax \gap(\kmax l + k_0)
    &\leq 4 \frac1N \sum_{n=1}^N \gap(n)\label{grid gap}
  \end{align}
  as in \cite{LauOtt15} we furthermore have
  \begin{align}
   \frac1\lmax \sum_{l=1}^\lmax \left(\lambda_{\kmax l + k_0} -\tfrac12\right)^2
    &\leq 4 \frac1N \sum_{n=1}^N \left(\lambda_{n} -\tfrac12\right)^2\quad  \text{and}
    \label{grid lambda av}\\
    \left| \lambda_{\kmax l + k_0} -\tfrac12\right| &\leq \tfrac18\quad(1\leq l \leq \lmax).
    \label{grid lambda}
  \end{align}
  We give a short counting argument for this. 
  By Proposition \ref{L2 bounds lambda}
  \begin{align*}
    &\# \left\{
    k_0\colon \text{(\ref{grid gap}) is violated, (\ref{grid lambda av}) is violated, or (\ref{grid lambda}) 
     is violated for some }l\right\}\\
    &\leq  \# 
    \left\{
    k_0\colon 
    \text{(\ref{grid gap}) is violated}
    \right\}
    +\# \left\{
    k_0\colon 
    \text{(\ref{grid lambda av}) is violated}
    \right\}
    +\sum_{l=1}^\lmax \# \left\{
    k_0\colon
    \text{(\ref{grid lambda}) 
    is violated for $l$}
    \right\}\\
    &\leq \frac\kmax4 + \frac\kmax4 +
    8^2 \sum_{l=1}^\lmax \sum_{k=1}^\kmax \left( \lambda_{\kmax l +k} -\tfrac12\right)^2
    \leq \frac\kmax2 + C
  \end{align*}
  for some constant $C=C(d,E_0,T)$.
  Therefore, we can adapt the proof of Proposition 4.1 in \cite{LauOtt15} so that indeed
  we can link the first variation of the dissipation with the normal velocity.
  Furthermore, the localization argument in Section 5 in \cite{LauOtt15} applies one-to-one
  so that we have (\ref{eq_deltaD}). 
\end{proof}

One of the main tools of the proof in \cite{LauOtt15} are Lemma 4.2 and its rescaled version, Corollary 4.3 in \cite{LauOtt15}.
Roughly speaking, this lemma establishes control over the distance of the super level sets $\{u>\tfrac12\}$ and $\{\tilde u>\tfrac12\}$
in terms of the $L^2$-distance of two functions $u, \tilde u\colon \R\to \R$ , provided at least one of the two functions is 
sufficiently monotone around the threshold value $\tfrac12$, which is measured by the term $\frac{1}{\sqrt h}\int_{ \frac13\leq u \leq \frac23 }
\! \left( \sqrt{h}\,\partial_1 u -\overline c \right)_-^2$; see Lemma \ref{1d lemma} below for the precise statement with more general threshold values, 
which however reduces to the statement in \cite{LauOtt15} when $\lambda = \tilde \lambda =\tfrac12$.
Note that such an estimate would clearly fail without such an extra term on the right-hand side. 

In order to motivate the lemma let us streamline its application to the thresholding scheme.
To this purpose let us ignore the localization $\eta$.
We apply the one-dimensional estimate to the thresholding scheme in a fixed direction $\nu^\ast\in S^{d-1}$ with  $\chi=\chi^h(t)$ and $\tilde \chi=\chi^h(t+\tau)$ for some $\tau =\alpha \h$.
 We think of the fudge factor $\alpha$ as small, but independent of $h$.
After dividing by $\alpha$ and integrating the resulting estimate over the further $d-1$ directions and over the time variable we obtain an estimate for the difference quotient
$\iint \left|\partial_t^\tau \chi^h\right|dx\,dt$ in terms of $\iint \h ( G_{h/2} \ast\partial_t^\tau \chi^h)^2 dx\,dt$, 
the above term measuring the monotonicity of $G_h \ast \chi^h(t-h)$ in direction $\nu^\ast$ and a term involving the $L^2$-norm of $\lambda_h-\frac12$.
The constant $\overline c$ in the term measuring the monotonicity is chosen such that if $\chi^h$ was a half space in direction $\nu^\ast$ this term would vanish. 
One can indeed prove, cf.\ Lemma 4.4 in \cite{LauOtt15}, that this term is bounded by the energy-excess 
\[\varepsilon^2:=\int_0^T \, E_h(\chi^h) - E_h(\chi^\ast)\,dt, \quad \text{for some half space } \chi^\ast \text{ in direction }\nu^\ast.\]
This term in turn is small (after
localization) by our strengthened convergence \eqref{conv_ass} and the local flatness of the limit --- which is guaranteed by De Giorgi's Structure Theorem.
The second term, $\iint \h ( G_{h/2} \ast\partial_t^\tau \chi^h)^2 dx\,dt$, is bounded by the dissipation and is thus finite by the energy-dissipation estimate \eqref{ED-estimate}. 
Therefore we obtain the following estimate for the discrete time derivative
\begin{align*}
 \int_0^T\int \left|\partial_t^\tau \chi^h\right|dx\,dt \lesssim \frac1{\alpha} \left( \varepsilon^2 +sT\right) 
 + \frac1{s^2} \alpha^2 E_0
 +\frac{1}{\alpha s^2} \frac1\h \int_0^T \left(\lambda_h-\frac12\right)^2 dt,
\end{align*}
which differs from the estimate in \cite{LauOtt15} only by the last right-hand side term involving the threshold value.
However, this term is of order $\h$ by our $L^2$-estimate, cf.\ Proposition \ref{L2 bounds lambda}.
We apply a localized version of this estimate and sum over a partition of unity with fineness $r>0$. Sending first $h$ to zero, the first right-hand side term converges to the the energy-excess 
on each patch, while the other terms stay uniformly bounded in $r$ if the patches have finite overlap. Then we take the limit $r\to0$ so that the first right-hand side term vanishes by De Giorgi's Structure Theorem.
Optimizing the additional parameter $s$ and then sending $\alpha$ to zero, the right-hand side stays uniformly bounded.
The resulting estimate resembles
\[
 \int_0^T\int \left|\partial_t^\tau \chi^h\right|dx\,dt  = O(1) \quad \text{for } \tau = o(\h).
\]
In comparison, the analogous estimate coming from \eqref{compactness dtchi} only holds for larger time scales $\tau \sim \h$.


\begin{lem}\label{1d lemma}
Let $u,\, \tilde u\in C^\infty (\R)$, $|\lambda-\tfrac12|<\tfrac18$
$\chi= \chara_{u> \lambda}$, $\tilde \chi= \chara_{\tilde u> \tilde \lambda}$ and $\eta\in C_0^\infty(-2r,2r)$ a
radially non-increasing cut-off for $(-r,r)$ inside $(-2r,2r)$.
Then
\begin{align*}
\frac{1}{\sqrt h}\int \!\eta \left|\chi-\tilde \chi\right| 
\lesssim \frac{1}{\sqrt h}\int_{ \frac13\leq u \leq \frac23 }
\!\!\eta \left( \sqrt{h}\,\partial_1 u -\overline c \right)_-^2  \!+ s
 +\frac1{s^2}  \frac{1}{\sqrt h}\int\!\eta \left(u-\tilde u \right)^2
+ \frac r{s^2}\frac{\big(\lambda-\tilde \lambda\big)^2}{\h}
\end{align*}
for any $s\ll 1$.
\end{lem}
\begin{proof}[Proof of Lemma \ref{1d lemma}]
The lemma follows from Corollary 4.3 in \cite{LauOtt15} with a shifting argument to make the
threshold value $\lambda$ appear.
Set $v := u -\lambda +\frac12$ so that $\chi =\chara_{v > \frac12}$ (and analogously with $\tilde v$) and Corollary 4.3 in \cite{LauOtt15} applies for
$v,\,\tilde v$: For any $s>0$, we have
\begin{align}\label{1dlem_proof}
 \frac{1}{\sqrt h}\int \eta |\chi-\tilde \chi| 
\lesssim & \frac{1}{\sqrt h}\int_{| v-\frac12| \leq s }
\eta \left( \sqrt{h}\,\partial_1 v -\overline c \right)_-^2 + s
 +\frac1{s^2}  \frac{1}{\sqrt h}\int \eta \left(v- \tilde v \right)^2.
\end{align}
Now we can resubstitute $v = u -\lambda +\tfrac12$ and $\tilde v = \tilde u -\tilde \lambda +\tfrac12$ on the right-hand side.
Then the integrand of the first integral stays unchanged since $\lambda$ is constant. 
If $|\lambda-\tfrac12|<\tfrac18$ and $s\ll 1$, the domain of integration is
\begin{align*}
 \big\{\big|v -\tfrac12\big|<s\big\} =  
 \{| u -\lambda|<s\} \subset \big\{\tfrac13 < u < \tfrac23\big\}.
\end{align*}
Since
$
 (v -\tilde v)^2
 \lesssim ( u - \tilde  u)^2
 +(\lambda - \tilde \lambda)^2,
$
also the second integral is in the form of the claim.
\end{proof}

%
%
%
%
%
%
%
\section{Mean-curvature flow with external force}\label{sec:force}
The following algorithm is based on an idea of Mascarenhas in \cite{mascarenhas1992diffusion} but we allow the forcing term to be space-time dependent.
\subsection{Algorithm and main result}
\begin{algo}\label{MBO_force}
Given the phase $\Omega$ at time $t=(n-1)h$, 
obtain the evolved phase $\Omega'$ at time $t=nh$ by:
\begin{enumerate}
\vspace{-3pt}
 \item Convolution step:
$
 \phi := G_h\ast \chara_{\Omega}.
$
\item Thresholding step:
$
  \Omega' : =  \{ \phi > \tfrac12 - \tfrac1{2\sqrt{\pi}}f(x,nh) \h \}.
$
\end{enumerate}
\end{algo}

The following weak formulation of mean-curvature flow with an external force has already been introduced in \cite{LucStu95}.
\begin{defn}[Motion by mean curvature with external force]\label{def_motion_by_mean_curvature+f}
 We say that
$
\chi: (0,T)\times \R^d
\to \{0,1\} 
$
\emph{moves by mean curvature with external force $f\in C^\infty([0,T]\times \R^d)$ and initial data $\chi^0$}
if there exists a function
$V\colon (0,T)\times\R^d\to \R$ with $V\in L^2(\left| \nabla \chi\right|dt)$, which is the normal velocity in the sense of (\ref{v=dtX}), such that
\begin{align}\label{H=v+f}
\int_0^T \int \left( \nabla\cdot \xi - \nu\cdot \nabla \xi \, \nu\right) \left|\nabla\chi\right|  dt
 = \int_0^T \int \left(V - f\right) \, \xi \cdot \nu \left| \nabla \chi \right| dt
\end{align}
for any $\xi \in C_0^\infty((0,T)\times \R^d,\R^d)$.
\end{defn}

 It is easy to see that also Algorithm \ref{MBO_force} can be interpreted as a minimizing movements scheme.
In fact, as in Lemma \ref{la_MM} we add a linear functional as a correction.
\begin{lem}[Minimizing movements interpretation]\label{force_la_MM}
 Given $\chi^0\in \{0,1\}$, let $\chi^1$ be obtained by Algorithm \ref{MBO_force}.
 Then $\chi^1$ solves
\begin{align}\label{force_MM_lagrange}
 \min \quad  E_h( \chi) + D_h(\chi-\chi_0) 
- \frac1{\sqrt{\pi}}\int f(nh,x) \, \chi\,dx,
\end{align}
where the minimum runs over all $\chi\colon \R^d \to \{0,1\}.$
\end{lem}

\begin{cor}[Euler-Lagrange equation]\label{cor_ELG-forcing}
Given $\chi^0\in \{0,1\}$, let $\chi^1$ be obtained by Algorithm \ref{MBO_force}.
Then $\chi^1$ solves the Euler-Lagrange equation
\begin{align}\label{ELG-2}
 \delta E_h (\chi^1,\xi)+\delta D_h (\,\cdot\, - \chi^0)(\chi^1,\xi) 
- \frac{1}{\sqrt{\pi}}\int \nabla\cdot\left(f(nh,x) \, \xi \right) \chi^1 \,dx =0.
\end{align}
\end{cor}

We can prove a conditional convergence result for Algorithm \ref{MBO_force} under the same assumption as in Section \ref{sec:vol}.
\begin{thm}\label{thm2}
Let $T<\infty$, $\chi^0\in\{0,1\}$ with $E(\chi^0) <\infty$ and $\{\chi^0=1\} \subset \subset \R^d$ and $f\in C^\infty([0,T]\times \R^d)$.
After passage to a subsequence, the functions $\chi^h$ obtained by Algorithm \ref{MBO_force}
converge to a function $\chi$ in $L^1((0,T)\times \R^d)$. Under the convergence assumption
(\ref{conv_ass}), $\chi$ moves by mean curvature with external force $f$
in the sense of Definition \ref{def_motion_by_mean_curvature+f}.
\end{thm}
We follow the same strategy as in Section \ref{sec:vol} to prove the theorem.
From the Euler-Lagrange equation (\ref{ELG-2}),
the mean curvature and normal velocity will be recovered from the limits of the first
variations of the energy and dissipation, respectively.  The convergence of the third term
in this algorithm is much easier.  $f$ is a smooth function in time and space so the convergence
of the third term is an immediate consequence of the compactness of the $\chi^h$ (c.f. 
Prop. \ref{force_compact}).
As before we write $A\lesssim B$ if there exists a constant $C=C(d)<\infty$ such that $A\leq C B$ and note that we also have (\ref{Eh<E}).

\subsection{Compactness}
Since there are no `bad' iterations as in Section \ref{sec:vol}, the argument in Lemma \ref{lem_good_iteration} yields the following linear growth estimate
and is sufficient to prove the boundedness of the sets.
Here we even have the optimal growth rate of the radii w.\ r.\ t. the time horizon $T$.
\begin{prop}\label{force_tightness}
There exists a universal constant $C<\infty$ such that for any $n=1,\dots,N$
\begin{align*}
 R_n \leq R_{n-1} + C h \|f\|_{\infty}.
\end{align*}
In particular, if $\Omega^0\subset B_R$ and the sets $\Omega^h(t)$ are obtained by Algorithm \ref{MBO_force},
then $\Omega^h(t) \subset B_{R^\ast} $ for all $t\leq T$, where $R^\ast =   R(1+CT\|f\|_\infty)$ for some universal constant $C<\infty.$
\end{prop}
The following lemma states the a priori estimate coming from the minimizing movements interpretation.
Here, we obtain extra terms coming from the forcing term which did not appear in Section \ref{sec:vol} due to the special structure of the equation there.
\begin{lem}[Energy-dissipation estimate]
 The approximate solutions $\chi^h$ constructed in Algorithm \ref{MBO_force} satisfy
\begin{align}\label{force_ED-estimate}
E_h(\chi^N) + \sum_{n=1}^N D_h(\chi^n-\chi^{n-1}) \leq E_0 
+ C \left(\|f\|_{\infty}+  \int_0^T \int \left|\partial_t f\right| dx\,dt \right).
\end{align}
\end{lem}
\begin{proof}
 Comparing $\chi^n$ to $\chi^{n-1}$, we have
 \begin{align*}
  E_h(\chi^n) + D_h(\chi^n-\chi^{n-1}) - \frac1{\sqrt{\pi}}\int f( nh) \, \chi^n\,dx 
  \leq E_h(\chi^{n-1}) - \frac1{\sqrt{\pi}}\int f(nh) \, \chi^{n-1}\,dx. 
 \end{align*}
 Iterating this estimate yields
 \begin{align}\label{force_comparing_for_edestimate}
  E_h(\chi^N) + \sum_{n=1}^N D_h(\chi^n-\chi^{n-1})
  \leq E_h(\chi^0) + \frac1{\sqrt{\pi}} \sum_{n=1}^N \int f(nh) \left(\chi^n- \chi^{n-1}\right) dx.
 \end{align}
 We handle the second right-hand side term by a discrete integration by parts,
 \begin{align*}
  \sum_{n=1}^N f(nh) \left(\chi^n- \chi^{n-1}\right)
  =f(Nh)\,\chi^N - f(0)\,\chi^0  -\sum_{n=1}^{N}  \left(f(nh)-f((n-1)h)\right) \chi^{n-1},
 \end{align*}
so that by Proposition \ref{force_tightness}
the right-hand side of (\ref{force_comparing_for_edestimate}) is estimated by
\begin{align*}
 E_0 + \frac1{\sqrt{\pi}}\|f\|_{\infty} \int \left( \chi^0 + \chi^N\right) dx 
 + \frac1{\sqrt{\pi}} \int_0^T \int \left|\partial_t f\right| dx\,dt
 \lesssim E_0 + \|f\|_{\infty} + \int_0^T \int \left|\partial_t f\right| dx\,dt,
\end{align*}
which concludes the proof.
\end{proof}
Now we can apply the same argument as in Section \ref{sec:vol} to prove the relative compactness of the approximate solutions.
\begin{prop}[Compactness]\label{force_compact}
Let $T<\infty$ and $\chi^0\in\{0,1\}$ with $E(\chi^0)<\infty$.
Then there exists a subsequence $h\searrow 0$ and a function $\chi\in \{0,1\}$ 
such that $\chi^h \to \chi$ in $L^1((0,T)\times \R^d)$ and the convergence holds almost everywhere in $(0,T)\times \R^d$.
\end{prop}
\subsection{Convergence}
\begin{proof}[Proof of Theorem \ref{thm2}]
By Proposition \ref{force_compact} we have compactness. Our a priori estimate \eqref{force_ED-estimate} and the strengthened convergence \eqref{conv_ass} allow
us to proceed as in Step 1 of the proof of Theorem \ref{thm1} above to construct the normal velocity and establish the integrability.

As in Section \ref{sec:vol}, we can apply Proposition \ref{deltaE} because of our strengthened convergence \eqref{conv_ass} so that we recover the mean curvature from the first variation of the energy.
To prove the analogue of Proposition \ref{deltaD}, i.\ e.\ convergence of the first variation of the dissipation
towards $\int V \,\xi \cdot \nu \left| \nabla \chi \right|$ we use Lemma \ref{1d lemma+f} below to apply the proof in \cite{LauOtt15}.
This turns out to be easier compared to the proof in Section \ref{sec:vol} since there are no `bad' iterations and we do not have to take special care of the shift of the mesoscopic time slices as in Step 2.
\end{proof}

The following lemma is the analogue of Lemma \ref{1d lemma} but adapted to to the setting of this problem.
There are two major differences.
On the one hand, here the threshold values are not constant in space so that we obtain an extra term coming from the first right-hand side integral 
in (\ref{1dlem_proof}) which gives an error term measuring the spatial variation of $f$.
But on the other hand, the mild bound on the threshold value, $|\lambda-\tfrac12|< \tfrac18$ in Lemma \ref{1d lemma}, is here automatically satisfied if the time step
$h$ is small enough.

\begin{lem}\label{1d lemma+f}
Let $u,\, \tilde u, \, f, \, \tilde f \in C^\infty (\R)$,
$\chi= \chara_{u> \frac12 - \frac1{2\sqrt{\pi}}f\h}$, $\tilde \chi= \chara_{\tilde u> \frac12 - \frac1{2\sqrt{\pi}}\tilde f\h}$ and let furthermore $\eta\in C_0^\infty(-2r,2r)$ be a
radially non-increasing cut-off for $(-r,r)$ inside $(-2r,2r)$.
Then
\begin{align*}
\frac{1}{\sqrt h}\int \eta |\chi-\tilde \chi| 
\lesssim \frac{1}{\sqrt h}\int_{ \frac13\leq u \leq \frac23 }
\eta \left( \sqrt{h}\,\partial_1 u -\overline c \right)_-^2 &+ s
 +\frac1{s^2}  \frac{1}{\sqrt h}\int \eta \left(u-\tilde u \right)^2\\
& + \frac r{s^2}\h \big(f-\tilde f\big)^2 + \h^3 \int \eta \left(\partial_1 f \right)^2
\end{align*}
for any $ s\ll 1$ and $h \ll \frac1{\|f\|^2_\infty}$.
\end{lem}

We conclude this section with a short remark on the necessary regularity of $f$.  In Theorem \ref{thm2} we assumed $f \in C^{\infty}\left( [0,T] \times \mathbb{R}^d \right)$.  
However this regularity assumption can be weakened.  Indeed, our proof of Theorem \ref{thm2} only used $f\in L^\infty$, $\partial_t f\in L^1$ and $\nabla f \in L^2$.
%
%
%
%
%

\section{Grain growth in polycrystals}\label{sec:crystal}
In this section we present and study a thresholding algorithm for simulating grain growth in polycrystals including boundary effects.
Especially for thin films this is very important since then these effects become more important.
\subsection{Preliminaries}
The energy that we are interested in is the following weighted sum of interfacial energies
\begin{align}\label{graingrowth_Energy}
 E(\Omega_1,\dots,\Omega_\numphases) =  \sum_{i,j} \sigma_{ij} \left| \Sigma_{ij} \right| + 2\sigma_0 \left| \Sigma_{0}\right|,
\end{align}
where the phases $\Omega_1,\dots,\Omega_\numphases$ represent the different grains and are assumed to be closed, intersect only through their boundaries and
\begin{align*}
 \Sigma_{ij} := \partial \Omega_i \cap \partial \Omega_j,\quad \Sigma_0 := \partial \left(\Omega_1 \cup \dots \cup  \Omega_\numphases \right).
\end{align*}
The number $\sigma_{ij}$ is the surface tension between Phase $i$ and Phase $j$ and $\sigma_0$ the surface tension between the crystal and the air which is an additional modeling parameter.
The equation we want to study is the gradient flow of the energy (\ref{graingrowth_Energy}) subject to the volume constraint
\begin{align*}
 \left|\Omega_1 \cup \dots \cup  \Omega_\numphases  \right| = \text{constant}.
\end{align*}
In particular we analyze a thresholding algorithm (Algorithm \ref{MBO_graingrowth}) and in Theorem \ref{thm3} we prove a (conditional) convergence result 
for a very general class of surface tensions that has been introduced in \cite{EseOtt14}.
Esedo\u{g}lu and Otto showed that this class includes the 2-d and 3-d Read-Shockley formulas which are very prominent models for grain boundaries with a small mismatch in the angle.
As in \cite{LauOtt15}, we need slightly stronger assumptions for the convergence proof.
We ask the matrix $\sigma=(\sigma_{ij})_{ij=1}^\numphases$  of surface tensions to satisfy
\begin{align}\label{sigma>0}
 \sigma_{ii} = 0,\quad
 \sigma_{ji} = \sigma_{ij} >0 \text{ for all } i\neq j
\end{align}
and furthermore the following triangle inequality
\begin{align}\label{triangle_surf_tens}
 \sigma_{ij} < \sigma_{ik} + \sigma_{kj} \quad \text{for all pairwise different }
  i, j, k.
\end{align}
For the dynamics, it is natural to assume that there exists a positive constant $\underline \sigma>0$ such that
\begin{align}\label{sigma<0}
 \sigma \leq -\underline \sigma <0 \quad \text{on } (1,\dots,1)^\perp
\end{align}
as a bilinear form.
Given a matrix of surface tension $\sigma$, the only modeling assumption on the parameter $\sigma_0$, the surface tension between the crystal and the air, is the
the lower bound
\begin{align}\label{sigma>12sigmamax}
\sigma_0 > \frac12\max_{i,j}\sigma_{ij}.
 \end{align}
In the following, we will normalize this parameter $\sigma_0=1$ by rescaling the other surface tensions $\sigma_{ij} \mapsto \frac{\sigma_{ij}}{\sigma_0}$ so that this modeling assumption turns
into an additional assumption on the matrix of (normalized) surface tensions between the grains:
\begin{align}\label{sigma<2}
 \sigma_{ij} < 2 \quad \text{for all } i,j.
 \end{align}
Note that given this additional assumption, the extended matrix of surface tensions given by the $(\numphases+1)\times (\numphases +1 )$-block matrix
\begin{align}\label{extendsigma}
  \begin{pmatrix}
 0		&  1  	&\cdots &  1 \\
 1 		&       &  	&  \\
 \vdots    	&	&\sigma &    \\
   1  		&       &	&    
 \end{pmatrix}
\end{align}
satisfies all the assumptions mentioned before and in particular (\ref{sigma<0}) with $\underline \sigma$
replaced by $\underline \sigma \wedge 2$.
The resulting equation then becomes
\begin{align}\label{graingrowth_pde1}
 V_{ij}= H_{ij}
\end{align}
on the smooth part of the interface $\Sigma_{ij}$, ($i,j\geq 1$) and
\begin{align}\label{graingrowth_pde2}
 \sigma_{ij}\nu_{ij}(p)+ \sigma_{jk}\nu_{jk}(p) + \sigma_{ki}\nu_{ki}(p) = 0,
\end{align}
whenever $p$ is a triple junction between the phases $i,\,j$ and $k$, and
\begin{align}\label{graingrowth_pde3}
 V_{0} = H_{0} - \langle H_0 \rangle
\end{align}
on the smooth part of the outer boundary $\Sigma_0$.

\medskip

Esedo\u{g}lu and Otto showed in \cite{EseOtt14} that - up to a constant - the energy $E$ in (\ref{graingrowth_Energy}) can be approximated by
\begin{align}\label{graingrowth_Eh}
 E_h(\chi) := \frac1\h\sum_{i,j\geq 1} \sigma_{ij} \int \chi_i\,G_h\ast \chi_j\,dx + \frac2\h \int \left(1-\chi_0 \right)G_h\ast \chi_0\,dx
\end{align}
for \emph{admissible} $\chi$, i.\ e.
\begin{align}\label{admissible}
 \chi= \left(\chi_0,\chi_1,\dots,\chi_\numphases \right)\colon \R^d \to \{0,1\}^{\numphases+1},\quad\textup{s.\ t. } \sum_{i=1}^\numphases \chi_i = 1-\chi_0.
\end{align}
Indeed, they proved that the functionals $E_h$ $\Gamma$-converge to $\frac{1}{\sqrt{\pi}}E$ as $h\to0$ when identifying the sets $\Omega_i$ with their characteristic functions
$\chi_i = \chara_{\Omega_i}$ and defining the area of the interface $\Sigma_{ij}$ between Phases $i$ and $j$ via the term
$
 \int\frac12\left(\left| \nabla \chi_i\right|+\left| \nabla \chi_i\right| - \left| \nabla \chi_i+\chi_j\right| \right)
$
so that the energy $E$ then becomes
\begin{align*}
 E(\chi) = \frac1{\sqrt{\pi}} \sum_{i,j\geq 1} \sigma_{ij}\int\frac12\left(\left| \nabla \chi_i\right|+\left| \nabla \chi_i\right| - \left| \nabla \chi_i+\chi_j\right| \right)
 + \frac{2}{\sqrt{\pi}}\int\left| \nabla \chi_0\right|.
\end{align*}

In the following we will w.\ l.\ o.\ g. assume that the total volume of the crystal is normalized to $1$, i.\ e.
\begin{align*}
 \left|\Omega_1 \cup \dots \cup  \Omega_\numphases  \right| = 1.
\end{align*}
  
\subsection{Algorithm and notation}
The following algorithm was proposed in \cite{An15} to model grain growth in thin polycrystals. Similar to Algorithm \ref{MBO_volume},
here the total volume of the polycrystal is preserved by the right choice of the threshold value.
\begin{algo}\label{MBO_graingrowth}
Given the phases $\Omega_1,\dots, \Omega_\numphases$ with total volume $1$ at time $t=(n-1)h$ and write
$\Omega_0 := \R^{d} \setminus \left( \Omega_1 \cup \dots \cup  \Omega_\numphases\right)$, 
obtain the evolved phases $\Omega_1',\dots, \Omega_\numphases'$ at time $t=nh$ by:
\begin{enumerate}
\vspace{-3pt}
 \item Convolution step:
\begin{align*}
  \phi_0 := G_h \ast  \Bigg( \sum_{j\geq 1} \chara_{\Omega_j} \Bigg),\quad \phi_i := G_h\ast \Bigg( \sum_{j\geq 1} \sigma_{ij}\chara_{\Omega_j} 
  +  \chara_{\Omega_0} \Bigg), \quad i \geq 1.
\end{align*}
\item Defining threshold value:
 Find $\lambda$ such that
\begin{align*}
 \left| \bigcup_{i \geq 1} \left\{  \phi_i < \phi_0 + \lambda 
 \right\} \right|
= 1.
\end{align*}
\item Thresholding step:
For $i=1,\dots, \numphases$ set 
\begin{align*}
     \Omega_i' := \{\phi_i < \phi_j \text{ for all } j\neq i, \, j\geq 1 \} \cap
     \{\phi_i < \phi_0 + \lambda \}
    \end{align*}
and $\Omega_0' := \R^d \setminus (\Omega'_1 \cup \dots \cup \Omega'_\numphases)$.
\end{enumerate}
\end{algo}

%
%
%

\subsection{Minimizing movements interpretation}

%
%
With a similar argument as before, using the linearized energy
\begin{align}\label{graingrowth_def_L}
 L_h(\phi,\chi) := \frac2\h \sum_{i=0}^\numphases \int \chi_i\, \phi_i \, dx,
\end{align}
we can interpret Algorithm \ref{MBO_graingrowth} as a minimizing movements scheme for the approximate energies $E_h$ defined in (\ref{graingrowth_Eh})
and dissipation $-E_h(\omega)$. 
Here the matrix of surface tensions $\sigma_{ij}$ is extended as in (\ref{extendsigma}). 
\begin{lem}[Minimizing movements interpretation]\label{graingrowth_la_MM}
 Given any admissible $\chi^0$, let $\phi,$ $\lambda$ and $\chi^1$ be obtained by Algorithm \ref{MBO_graingrowth}.
 Then $\chi^1$ solves
\begin{align}\label{graingrowth_MM_lagrange}
 \min \quad  E_h( \chi) - E_h(\chi-\chi^0) 
- \frac{2\lambda}{\h} \int (1-\chi_0) \,dx,
\end{align}
where the minimum runs over (\ref{admissible}).
Or equivalently,
\begin{align}\label{graingrowth_MM_constraint}
 \min \quad  E_h( \chi) - E_h(\chi-\chi^0) \quad \textup{s.\ t.}
\quad  \int (1-\chi_0) \,dx= 1,
\end{align}
where the minimum runs over (\ref{admissible}) and is additionally constrained by the volume constraint.
\end{lem}
\begin{proof}
 Indeed, for any admissible $\chi$ in the sense of (\ref{admissible}) we have
 \begin{align*}
    \sum_{i=0}^\numphases \chi_i\,\phi_i - \lambda\left( 1-\chi_0\right) 
    = \chi_0 \left( \phi_0+\lambda\right) + \sum_{i=1}^\numphases \chi_i\,\phi_i - \lambda 
    \overset{(\ref{admissible})}{\geq} \min\left\{ \phi_0+\lambda, \phi_1,\dots,\phi_\numphases\right\} - \lambda.
 \end{align*}
For $\chi^1$ obtained by Algorithm \ref{MBO_graingrowth} in turn we have equality in the above inequality so that $\chi^1$ minimizes the left-hand side pointwise.
In particular, after integration we see that $\chi^1$ minimizes the functional
\begin{align*}
 \frac2\h \sum_{i=0}^\numphases \int \chi_i\, \phi_i \, dx  - \frac{2\lambda}\h \int \left( 1-\chi_0\right)dx = L_h(\chi,\phi) - \frac{2\lambda}\h \int \left( 1-\chi_0\right)dx.
\end{align*}
By the quadratic nature of the functional $E_h$ we have
\begin{align*}
  L_h(\phi,\chi) = E_h( \chi) - E_h(\chi-\chi_0) + \text{Terms depending only on } \chi^0,
\end{align*}
which proves the first claim (\ref{graingrowth_MM_lagrange}).
Since the last term in (\ref{graingrowth_MM_lagrange}) is constant for $\chi$ with the volume constraint, we also have (\ref{graingrowth_MM_constraint}).
 \end{proof}
Again, as a direct consequence of the minimizing movements interpretation, we obtain an a priori estimate by comparing the solution to its predecessor.
\begin{lem}[Energy-dissipation estimate]
The approximate solutions $\chi^h$ satisfy
\begin{align}\label{graingrowth_ED-estimate}
E_h(\chi^N) - \sum_{n=1}^N E_h(\chi^n-\chi^{n-1}) \leq E_0.
\end{align}
\end{lem}
Note that our assumption \eqref{sigma<0} guarantees that $\sqrt{-E_h}$ defines a norm on the process space $\{\omega \colon \sum_i \omega_i=0\}$ in the same spirit as 
$\sqrt{D_h}$ in the previous two sections.
%
%
\begin{defn}[First variation]
 For any admissible $\chi\in\{0,1\}^{\numphases}$ and $\xi \in C_0^\infty(D,\R^d)$ let $\chi_s$ be generated by the flow of $\xi$, i.\ e. $\chi_{i,s}$
 solves the following distributional equation:
\begin{align*}
 \partial_s \chi_{i,s} + \xi \cdot \nabla \chi_{i,s} =0.
\end{align*}
We denote the first variation along this flow by
\begin{align*}
 \delta E_h (\chi,\xi):= \frac{d}{ds} E_h(\chi_s)\big|_{s=0},\quad 
\delta E_h (\,\cdot\, - \tilde \chi)(\chi,\xi):= \frac{d}{ds} E_h(\chi_s-\tilde \chi)\big|_{s=0},
\end{align*}
where $\tilde \chi\in \{0,1\}$ is fixed.
\end{defn}
%
%
\begin{cor}[Euler-Lagrange equation]\label{cor_ELG_graingrowth}
Given an admissible $\chi^0\in\{0,1\}^{\numphases}$, let $\chi^1$ be obtained by Algorithm \ref{MBO_graingrowth} with threshold value $\lambda$.
Then $\chi^1$ solves the Euler-Lagrange equation
\begin{align}\label{ELG_graingrowth}
 \delta E_h (\chi^1,\xi)-\delta E_h (\,\cdot\, - \chi^0)(\chi^1,\xi) 
-\frac{2\lambda}{\h}\int\left(\nabla\cdot \xi\right) \left(1-\chi^1_0\right)dx =0.
\end{align}
\end{cor}
The idea underlying the convergence proof now follows the framework laid out in Section \ref{sec:vol}.  
The first variation of the approximate energy will be shown to converge 
to the mean curvature of the crystal/grain boundary in a weak sense.  
The first variation of the dissipation will be shown to converge to the
velocity in a weak sense.  And the first variation of the Lagrange multiplier term will converge to zero on the inner
grain boundaries, and the average of the mean curvature over the outer solid-vapor interface.  
The precise limit is formulated in the next definition.

The following definition is similar to the notion for multi-phase mean-curvature flow as described in \cite{LauOtt15} but incorporates
an additional constraint on the total volume. 
\begin{defn}\label{def_graingrowth}
Fix some finite time horizon $T<\infty$, a matrix of surface tensions $\sigma$ as above and initial data 
$\chi^0\colon \R^d \to \{0,1\}^{\numphases}$ with $E_0 := E(\chi^0) <\infty$.
We say that
\begin{align*}
\chi = \left(\chi_1,\dots,\chi_\numphases\right): (0,T)\times \R^d 
\to \{0,1\}^{\numphases}
\end{align*}
with $\chi_0 := 1- \sum_i \chi_i \in \{0,1\}$ a.\ e. and $\chi(t)\in BV(\R^d,\{0,1\}^{\numphases})$ for a.\ e. $t$
moves by \emph{total-volume preserving mean-curvature flow}
if 
\begin{align}\label{graingrowth_H=v}
\sum_{i,j =1}^\numphases \sigma_{ij} &\int_0^T \int \left(\nabla\cdot\xi - \nu_i \cdot \nabla \xi \,\nu_i
- \xi \cdot \nu_i\, V_i \right) 
\left(\left|\nabla\chi_i\right| + \left|\nabla\chi_j\right| - \left|\nabla(\chi_i+\chi_j)\right|\right)  dt \notag\\
+ 2&\int_0^T \int \left(\nabla\cdot\xi - \nu_0 \cdot \nabla \xi \,\nu_0
- \xi \cdot \nu_0 \left( V_0 + \Lambda\right) \right) \left| \nabla \chi_0\right|  dt
= 0
\end{align}
for all $\xi\in C^\infty_0((0,T)\times \R^d, \R^d)$, where the functions $V_i\colon (0,T)\times\R^d\to \R$ are normal velocities in the sense that
\begin{align}\label{graingrowth_v=dtX}
\int_0^T \int \partial_t \zeta \, \chi_i\, dx\,dt + \int \zeta(0) \chi_i^0 \,dx
= - \int_0^T \int \zeta\, V_i\left| \nabla \chi_i\right|dt
\end{align}
for all $\zeta\in C^\infty([0,T]\times \R^d)$ with $\zeta(T)=0$ and $\supp \zeta(t) \subset\subset \R^d$ and 
all $i\in \left\{0,1,\dots,\numphases\right\}$ and if the Lagrange multiplier $\Lambda\colon (0,T)\to \R$ is such that
the volume of the solid phase $(1-\chi^0)$ is preserved:
\begin{align}\label{graingrowth_volume_preserved}
 \sum_{i=1}^\numphases \int \chi_i(t) \,dx = \text{constant}.
\end{align}
\end{defn}
 
\begin{rem}
 We assume the following convergence of the energies defined in (\ref{graingrowth_Eh}).
\begin{align}\label{conv_ass_graingrowth}
 \int_0^T E_h(\chi^h)\,dt \to \int_0^T E(\chi)\,dt.
\end{align}
\end{rem}

\subsection{Main result} 
\begin{thm}\label{thm3}
 Let $T<\infty$ be a finite time horizon, $\chi^0= (\chi_1^0,\dots, \chi_\numphases^0)$ be admissible initial data with $E(\chi^0)<\infty$ and
 $\{\sum_i\chi_i^0=1\} \subset \subset \R^d$ and let the matrix of surface tensions $\sigma$ satisfy the assumptions (\ref{sigma>0})-(\ref{sigma>12sigmamax}).
 After passage to a subsequence, the approximate solutions $\chi^h$ constructed in Algorithm \ref{MBO_graingrowth} converge to an admissible $\chi$ in $L^1((0,T)\times \R^d)$.
 Given the convergence assumption (\ref{conv_ass_graingrowth}), $\chi$ moves by total-volume preserving mean-curvature flow according to Definition \ref{def_graingrowth}.
\end{thm}

One of the main ingredients -- as in Section \ref{sec:vol} -- is the following estimate on the Lagrange multiplier.

%
%
\begin{prop}\label{graingrowth_L2 bounds lambda}
 Let $\chi_0$ be admissible. Given the approximate solutions $\chi^h$ obtained by Algorithm
\ref{MBO_volume} with thresholding values $\lambda_h$, we have the estimate
\begin{align*}
 \int_0^T  \lambda_h^2 \, dt \lesssim \left(1+T\right)\left( 1+E_0^4\right) h.
\end{align*}
\end{prop}

\begin{proof}
We can adapt the proof of Proposition \ref{L2 bounds lambda}.
We square the Euler-Lagrange equation and obtain an equation similar to \eqref{ELGstepn} but with $\chi^n$ replaced by $1-\chi_0^n$ on the left-hand side.
The estimates on $\delta E$ and $\delta D$, i.\ e.\ Steps 1 and 2 work analogously with help of the a priori estimate \eqref{graingrowth_ED-estimate}.
In Step 3 we choose the test vector field $\xi$ to satisfy
\begin{align*}
 \int \left(\nabla\cdot \xi\right) \left(1-\chi_0\right) dx &\geq \frac 12\quad \text{and}\\
\left\|  \xi\right\|_{W^{1,\infty}}		&\lesssim 1+ E_0.
\end{align*}
The construction of $\xi$ is the same as there but with $\chi$ replaced by $1-\chi_0$, which has a fixed volume $\int(1-\chi_0)\,dx = 1$.
\end{proof}

\subsection{Compactness}
%
%
\begin{prop}[Compactness]\label{graingrowth_compactness}
There exists a subsequence $h\searrow 0$ and an admissible 
$\chi\in L^1((0,T)\times\R^d, \{0,1\}^{\numphases})$ such that
\begin{align}\label{graingrowth_comp_conv_1}
 \chi^h\longrightarrow \chi \quad \text{in } L^1( (0,T)\times\R^d).
\end{align}
Moreover,
\begin{align}\label{graingrowth_comp_conv_2}
 \chi^h\longrightarrow \chi \quad \text{a.\ e. in } (0,T)\times \R^d
\end{align}
and $\chi(t)\in BV(\R^d,\{0,1\}^{\numphases+1})$, $\int\left(1-\chi_0\right)dx =1 $ 
and $1-\chi_0\subset \subset \R^d$ for a.\ e. $t\in (0,T)$.
\end{prop}

As in Section \ref{sec:vol}, this follows from \cite{LauOtt15} and the following two lemmas, which
guarantee that the phases stay in a bounded region.
In the proofs, we will reduce the statements until we can apply Lemma \ref{lem_bad_iteration}
and Lemma \ref{lem_good_iteration}, respectively to conclude.

\begin{lem}\label{graingrowth_lem_bad_iteration}
$R_n$ may be chosen such that $R_n \leq 3 R_{n-1}$.
\end{lem}
\begin{proof}
 For the sake of notational simplicity we will assume w.\ l.\ o.\ g.\ $n=1$.
 We want to give a similar, energy-based argument as in the proof of Lemma \ref{lem_bad_iteration}. 
 Let $1-\chi_0^0 $, the crystal at time $0$, be located inside $B_R$.
 We write $\Omega_1,\dots,\Omega_\numphases$ for the update in Algorithm \ref{MBO_graingrowth},
 write $\chi_i = \chara_{\Omega_i}$ and assume that 
 $U := \Omega_1\setminus B_{3R}$ has positive volume and construct $\widetilde U \subset B_R $
 with the same volume as $U$ as in the proof of Lemma \ref{lem_bad_iteration}.
 Then we define the competitor 
 $\tilde \chi$ by setting $\widetilde \Omega_1 := ( \Omega_1 \setminus U) \cup \widetilde U $
 leaving the phases $\Omega_i$, $i\geq 2$ unchanged so that 
 $\widetilde \Omega_0 := ( \Omega_0 \setminus \widetilde U) \cup U$.
 Recalling the linearized energy defined in (\ref{graingrowth_def_L}), we see that
 \begin{align*}
  L_h(\phi,\chi) - L_h(\phi,\tilde \chi)
  = \frac2\h \int \left( \chi_0 - \tilde \chi_0\right) \phi_0 
  + \left( \chi_1 - \tilde \chi_1\right) \phi_1\,dx.
 \end{align*}
By construction we have $\chi_0-\tilde\chi_0 
= -( \chi_1 - \tilde \chi_1) = \chara_{\widetilde U}-\chara_{U}$.
Rewriting $\phi_1$ in the form
$$
\phi_1 = \left( 1- \sum_{j\geq 1} G_h\ast \chi_j^0 \right) 
+\sum_{j\geq 1} \sigma_{1j} \,G_h\ast \chi_j^0,
$$
we thus have
\begin{align*}
 L_h(\phi,\chi) - L_h(\phi,\tilde \chi)
 =& \frac2\h \int \left( \phi_0 -\phi_1\right) \left( \chara_{\widetilde U}-\chara_{U}\right) dx\\
 =& \frac2\h \sum_{j=1}^\numphases \left( 2 -\sigma_{1j}\right)
 \int G_h\ast \chi_{j}^0 \left( \chara_{\widetilde U}-\chara_{U}\right)  dx.
\end{align*}
Note that by the normalization (\ref{sigma<2}), which guarantees the strict triangle inequality for the extended surface tensions,
each prefactor in the sum is strictly positive, furthermore
we have (\ref{estimate on phi}) for $G_h\ast \chi_j^0$ playing the role of $\phi$ there and by construction of
$\widetilde U$ the right-hand side term is positive which gives the desired contradiction.
\end{proof}

\begin{lem}\label{graingrowth_lem_good_iteration}
 Over `good' iterations we have the estimate
\begin{align*}
 R_n \leq R_{n-1} + C \sqrt{h}| \lambda_n|.
\end{align*}
\end{lem}
\begin{proof}
 As before, we can ignore the index $n$ and set $n=1$ for convenience.
 Let $1-\chi_0^0 $, the crystal at time $0$, be located inside some ball $B_{R_0}$.
 As in the proof of Lemma \ref{lem_good_iteration}, via a comparison argument, we want to prove that
 $1-\chi_0 $, the crystal at time $h$, does not intersect the half space $\{x\cdot e > R_0+C\h\}$
 for any choice of $e\in \sphere$.
 That means, we want to prove the existence of a constant $C<\infty$ such that
 \begin{align*}
  \phi_0 + \lambda < \phi_i\quad \text{for all } i\geq 1 \quad \text{in }\{x\cdot e > R_0+C\h\}.
 \end{align*}
 By rotational symmetry we may again restrict to the case $e=e_1$.
 Since we may relabel the phases inside the crystal, we may also prove the inequality only for $i=1$.
 In that case, writing $x=(x_1,x')\in \R^d$, we have
 \begin{align*}
  \left(\phi_0 - \phi_1 \right) (x) 
  \leq   G_h\ast\left(\sum_{i\geq 1}  \chi_{i}^0 - \chi_{0}^0\right).
 \end{align*}
 Thus, writing $\chi^0 := \sum_{i\geq 1}  \chi_{i}^0$, we reduced the problem to the two-phase analogue
 which we handled in Lemma \ref{lem_good_iteration}.
 Indeed, using the same comparison argument, i.\ e. using $\chi^0 \leq \chara_H$, where $H=\{x_1< R_0\}$
 is a half space tangent to $\partial B_{R_0}$ we find
 \begin{align*}
  \left(\phi_0 - \phi_1 \right) (x) \leq 2  \int_0^{x_1-R_0} G_h^1(z_1)\,dz_1.
 \end{align*}
 Since for a `good' iteration $\lambda$ is bounded, as in the proof 
 of Lemma \ref{lem_good_iteration} we can find a constant $C<\infty$,  so that
 \begin{align*}
  \left(\phi_0 - \phi_1 \right) (x) \leq 2   \frac{R_1-R_0}\h \min_{|z_1|\leq C} G^1(z_1) \leq |\lambda|
 \end{align*}
 which concludes the proof.
\end{proof}

 \subsection{Convergence}
 The following lemma is the main technical ingredient of the convergence proof.
 It is slightly more general than our set-up here since it allows for several Lagrange-multipliers so that
 the order parameter becomes $\sigma\, u + \lambda$ instead of $\sigma\,u$, where $u=G\ast \chi$ and $\lambda\in \R^\numphases$.
 The changes in the statement w.\ r.\ t. Lemma 4.5 in \cite{LauOtt15} are of the same form as before in Lemma \ref{1d lemma} except for
 a lower order term, $|\lambda|$, which can be absorbed by the term $ \frac{r}{s^2} \frac{|\lambda-\tilde \lambda|^2}{\h}$ and terms of order $\h$.
 \begin{lem}\label{graingrowth_1dlemma}
 Let $N\in \N$, $I\subset \R$ be an interval, $h>0$, $\eta\in C_0^\infty(\R)$, $0\leq \eta \leq 1$, radially non-increasing and
 $u,\, \tilde u \colon I \to \R^N$ be two maps into the standard simplex $\{U_i\geq 0, \sum_i U_i =1\}\subset \R^N$.
 Let $\sigma\in \R^{N\times N}$ be admissible in the sense of (\ref{sigma>0})-(\ref{sigma<0})
 and $\lambda,\,\tilde \lambda \in \R^N$ with $|\lambda| \leq \frac18$.
 Define $\phi:= \sigma\, u+\lambda $, 
 $\chi_i := \chara_{\{\phi_i > \phi_j \; \forall j\neq i\}}$ and $\tilde \phi,\, \tilde \chi_i $ in the same way.
 Then
\begin{align*}
  \frac{1}{\sqrt h}\int \eta \left|\chi-\tilde \chi\right|  dx_1
  \lesssim \frac1s \varepsilon^2 + s + |\lambda|
  +\frac1{s^2}  \frac{1}{\sqrt h}\int \eta \left| u -\tilde u\right|^2 dx_1 + \frac{r}{s^2} \frac{|\lambda-\tilde \lambda|^2}{\h}
\end{align*}
for $s\ll1$, where
\begin{align*}
 \varepsilon^2 := \frac1\h \int_{\frac13 \leq u_1 \leq \frac23} \left(\h \partial_1 u_1 - \overline c \right)_-^2 dx_1
 + \frac1\h \sum_{j\geq 3} \int \eta \left[ u_j \wedge (1-u_j)\right] dx_1.
\end{align*}
\end{lem}

\begin{proof}
 As in the proof of Lemma 4.5 in \cite{LauOtt15} by scaling we can assume $h=1$ and by taking convex combinations, we may assume
 $\eta=\chara_I$ for some interval $I\subset \R$:
 \begin{align*}
  \int_I\left|\chi-\tilde \chi\right| 
  \lesssim \int_{|u_1-\frac12| \leq s + |\lambda|} \left( \partial_1 u_1 - \overline c \right)_-^2 
 &+ \frac1s \sum_{j\geq 3} \int_I \left[ u_j \wedge (1-u_j)\right]  + s + |\lambda|\\
  &+\frac1{s^2}  \int_I \left| u -\tilde u\right|^2+ \frac{|I|}{s^2}|\lambda-\tilde\lambda|^2.
 \end{align*}
We will prove 
\begin{align}\label{graingrowth_1dlemma_setinclusion}
 \{\chi\neq \tilde \chi\} \subset \Big \{|u_1-\tfrac12| \lesssim s +|\lambda| \Big \}
 \cup \Big \{\sum_{j\geq  3}\left[ u_j \wedge(1-u_j) \right] \gtrsim s \Big \} \cup  \{|u-\tilde u| +|\lambda| \gtrsim s \}.
\end{align}
We fix $i\in\{1,\dots,\numphases\}$ and define
$
 v := \min_{j\neq i} \phi_j  - \phi_i
$
as in \cite{LauOtt15}.
Then $\chi_i= \chara_{v>0}$ and
\begin{align*}
  \{\chi_i\neq \tilde \chi_i\} \subset \{|v|<s\} \cup \{|v-\tilde v| \geq s\}.
\end{align*}
We clearly have
\begin{align*}
 |v-\tilde v| \lesssim |u-\tilde u| + |\lambda-\tilde \lambda|
\end{align*}
so that our goal is to prove
\begin{align}\label{graingrowth_1dlemma_inequalities}
 |u_1-\tfrac12| \lesssim s +|\lambda| \quad \text{or} \quad \sum_{j\geq  3}\left[ u_j \wedge(1-u_j) \right] \gtrsim s \quad \text{on } \{|v|<s\},
\end{align}
which then implies (\ref{graingrowth_1dlemma_setinclusion}).
In order to prove (\ref{graingrowth_1dlemma_inequalities}) we claim that
\begin{align}\label{graingrowth_1dlemma_inequalities2}
 u_j \leq \frac12 + \frac{s+|\lambda|}{\sigma_{\min}}\quad \text{on }\{|v|<s\}.
\end{align}
First we show that (\ref{graingrowth_1dlemma_inequalities2}) implies (\ref{graingrowth_1dlemma_inequalities}).
By (\ref{graingrowth_1dlemma_inequalities2}) we have on the one hand
\begin{align*}
 u_1 \leq \frac12 + C(s+|\lambda|)\quad \text{on } \{|v|<s\}
\end{align*}
and on $\{|v|<s\} \cup \{ u_1 \leq \frac12 -C(s+|\lambda|) \} $ we have
\begin{align*}
 \sum_{j\geq  3}\left[ u_j \wedge(1-u_j) \right]  = \sum_{j\geq  3} u_j - \sum_{j\geq  3}\left(1-2 u_j \right)_-
 \geq \left( C - \frac1{\sigma_{\min}} - 2 \numphases \frac1{\sigma_{\min}}\right) s
 \gtrsim s
\end{align*}
if $C<\infty$ is large enough.
This implies (\ref{graingrowth_1dlemma_inequalities}).

We are left with proving the inequality (\ref{graingrowth_1dlemma_inequalities2}).
As in \cite{LauOtt15} we decompose the set
\begin{align*}
\{\left| v\right|<s\} = 
\bigcup_{j \neq i} E_j,
\quad E_j := \big\{\left| \phi_i-\phi_j\right|<s,\, \phi_j = \min_{k \neq i} \phi_k\big\}.
\end{align*}
For $k\neq \{i,j\}$ by the triangle inequality for the surface tensions we have on $E_j$ 
\begin{align*}
 \phi_j \leq \phi_k \leq \sigma_{jk}\left(1-2u_k\right) + \phi_j + \lambda_k - \lambda_j,
\end{align*}
so that
\begin{align*}
 u_k \leq \frac12 + \frac{\lambda_k - \lambda_j}{2\sigma_{jk}}.
\end{align*}
For $u_i$ we can use that $\phi_j-s\leq \phi_i$ on $E_j$ so that using the same chain of inequalities we have
\begin{align*}
 u_i \leq \frac12 + \frac{s+\lambda_i -\lambda_j}{2\sigma_{ij}}.
\end{align*}
Since also $\phi_i-s \leq \phi_j$ on $E_j$ we have the analogous inequality for $u_j$, which concludes (\ref{graingrowth_1dlemma_inequalities2}).
\end{proof}

As in \cite{LauOtt15}, we have the following convergence of the first variations of the (approximate) energies.
\begin{prop}[Energy and mean curvature; Prop.\ 3.1 in \cite{LauOtt15}]\label{deltaE_graingrowth}
 Under the convergence assumption (\ref{conv_ass_graingrowth})
 \begin{align*}
 &\lim_{h\to0}\int_0^T \delta E_h(\chi^h,\xi)\,dt \\
 & \qquad \qquad =  \frac1{\sqrt{\pi}}\sum_{i,j =0}^\numphases \sigma_{ij} \int_0^T \int \left(\nabla\cdot\xi - \nu \cdot \nabla \xi \,\nu \right) 
 \frac12\left(\left|\nabla\chi_i\right| + \left|\nabla\chi_j\right| - \left|\nabla(\chi_i+\chi_j)\right|\right)  dt\\
 \end{align*}
 for any $\xi \in C_0^\infty((0,T)\times \R^d,\R^d).$
\end{prop}

Since we have both, the estimate on the Lagrange multiplier $\lambda$ in Proposition \ref{graingrowth_L2 bounds lambda}
and the important estimate Lemma \ref{graingrowth_1dlemma}, as in Section \ref{sec:vol}, we can adapt the techniques from \cite{LauOtt15} to recover the
normal velocity from the first variation of the dissipation functional.
\begin{prop}[Dissipation and normal velocity]\label{deltaD_graingrowth}
 There exist functions $V_i\colon (0,T)\times \R^d \to \R$ which are normal velocities in the sense of (\ref{graingrowth_v=dtX}).
 Given the convergence assumption (\ref{conv_ass_graingrowth}), $V_i\in L^2(\left| \nabla\chi\right|dt)$ and for any $\xi \in C_0^\infty((0,T)\times \R^d,\R^d)$ we have
 \begin{align*}
 \lefteqn{\lim_{h\to0}\int_0^T -\delta E_h(\,\cdot\,- \chi^h(t-h))(\chi^h(t),\xi)\,dt}\\
 &\qquad \qquad =   -\frac1{\sqrt{\pi}}\sum_{i,j =0}^\numphases \sigma_{ij} \int_0^T \int  \xi \cdot \nu_i\, V_i
\frac12\left(\left|\nabla\chi_i\right| + \left|\nabla\chi_j\right| - \left|\nabla(\chi_i+\chi_j)\right|\right)  dt.
 \end{align*}
\end{prop}
\begin{proof}
 \step{Step 1: Construction of the normal velocities and (\ref{graingrowth_v=dtX}).}
  As before in the two-phase case we can also adapt the proof of \cite{LauOtt15} in this case. Indeed, the argument there only makes use of the
  a priori estimate \eqref{graingrowth_ED-estimate} and the strengthened convergence \eqref{conv_ass_graingrowth}.
 \step{Step 2: Argument for (\ref{eq_deltaD}).}
 Our $L^2$-estimate on the Lagrange-multiplier $\lambda$ allows us to choose the shift of the mesoscopic time slices as in Step 2 of the proof of Proposition \ref{deltaD} such that 
 these slices are `good' in the sense that $|\lambda|\leq \frac18$.
 Now we may use our main technical ingredient, Lemma \ref{graingrowth_1dlemma}, for all mesoscopic time slices and hence we can apply the proof as in Section \ref{sec:vol} before.
\end{proof}
These two propositions conclude the proof of Theorem \ref{thm3}.

 \subsection*{Acknowledgements}
The authors want to thank Selim Esedo\u{g}lu and Felix Otto for fruitful discussions.  Additionally they would like to thank the referees for their helpful comments, which have improved the overall readability of the paper.
Finally, the second author would like to thank the MPI for their generous support in funding his travels and living accommodations while this research was undertaken.  
%
%
\printbibliography
\end{document}